\documentclass[11pt, a4paper]{article}

\usepackage{a4,amsmath,amssymb, amsthm, latexsym, color, graphicx,url}
\usepackage{tikz}
\usepackage{url}
\usetikzlibrary{arrows.meta}
\usetikzlibrary{decorations.markings}
\usetikzlibrary{calc}

\usepackage{authblk}
\usepackage{fullpage}
\usepackage{enumitem}
\usepackage{xcolor}
\usepackage{microtype}
\usepackage{listings}

\definecolor{dkgreen}{rgb}{0,0.6,0}
\definecolor{gray}{rgb}{0.5,0.5,0.5}
\definecolor{mauve}{rgb}{0.58,0,0.82}

\newtheorem{theorem}{Theorem}[section]
\newtheorem{proposition}[theorem]{Proposition}
\newtheorem{lemma}[theorem]{Lemma}
\newtheorem{corollary}[theorem]{Corollary}

\theoremstyle{definition}
\newtheorem{example}[theorem]{Example}
\newtheorem{definition}[theorem]{Definition}
\newtheorem{remark}[theorem]{Remark}
\newtheorem{construction}[theorem]{Construction}

\newcommand{\pg}{{\rm PG}}
\newcommand{\pgl}{{\rm PGL}}
\newcommand{\pgammal}{{\rm P}\Gamma{\rm L}}
\newcommand{\spec}{\operatorname{Spec}}

\newcommand{\Tr}{\operatorname{Tr}}
\newcommand{\F}{\mathbb{F}}

\title{The Intersection Distribution: New Results and Perspectives}
\author[1]{Sophie Huczynska}
\author[2] {Lukas Klawuhn}
\author[3] {Maura B. Paterson}
\affil[1]{School of Mathematics and Statistics, University of St Andrews, St Andrews, KY16 9SS, Scotland, UK}
\affil[2]{Department of Mathematics, Paderborn University, Warburger Str.\ 100, 33098 Paderborn, Germany}
\affil[3]{School of Computing and Mathematical Sciences, Birkbeck, University of London, Malet St, WC1E 7HX, UK}

\date{}

\begin{document}
\maketitle
\begin{abstract}
Intersection distribution and non-hitting index are concepts introduced recently by Li and Pott as a new way to view the behaviour of a collection of finite field polynomials.  With both an algebraic interpretation via the intersection of a polynomial with a set of lines, and a geometric interpretation via a $(q+1)$-set  possessing an internal nucleus, the concepts have proved their usefulness as a new way to view various long-standing problems, and have applications in areas such as Kakeya sets. In this paper, by exploiting connections with diverse areas including the theory of algebraic curves, cyclotomy and the enumeration of irreducible polynomials, we establish new results and resolve various Open Problems of Li and Pott.  We prove geometric results which shed new light on the relationship between intersection distribution and projective equivalence of polynomials, and algebraic results which describe and characterise the degree of $S_f$ - the index of the largest non-zero entry in the intersection distribution of  $f$.  We provide new insights into the non-hitting spectrum, and show the limitations of the non-hitting index as a tool for characterisation. Finally, the benefits provided by the connections to other areas are evidenced in two short new proofs of the cubic case.
\end{abstract}

\section{Introduction}\label{sec:introduction}
Polynomials over finite fields are studied in a range of contexts and have many applications, both mathematical and practical \cite{ffhandbook}. In \cite{LiPot}, Li and Pott introduce new tools for studying the behaviour of polynomials over finite fields, the {\em intersection distribution} and the {\em non-hitting index}. With two different interpretations (algebraic and geometric), this framework offers the opportunity to tap-into multiple different proof techniques to tackle problems which lie within its scope. It is directly connected to long-standing areas of interest, such as attempts to construct and classify o-polynomials, and results about the intersection distribution have been applied in the context of constructing Kakeya sets and Steiner triple systems \cite{KyuLiPot,LiPot}.

Let $q$ be a prime power, and let $f\in\F_q[x]$ be a polynomial over $\F_q$. We are interested in analysing the polynomial $f$ through intersections of its graph $\{(x,f(x)): x \in \F_q\}$ with lines in the classical affine plane; the intersection sizes provide useful information about $f$. In the affine setting, we will call lines of the form $x=c$ ``vertical" lines, and lines of the form $y=c$ ``horizontal" lines.  Since each vertical line intersects the graph of $f$ in precisely one point, these are ignored in the following definition, which considers only intersections of the graph with the $q^2$ non-vertical lines.  We count intersection points purely combinatorially without invoking any notion of algebraic multiplicity. 

\begin{definition}[\cite{LiPot}] \label{def:affine}
Let $f\in\F_q[x]$.  For $0\leq i \leq q$, let $v_i(f)$ denote the number of $(a,b) \in \F_q^2$ such that the equation $f(x)-ax-b=0$ has precisely $i$ solutions in $\F_q$.  The sequence $(v_i(f))_{i=0}^q$ is the {\em intersection distribution} of $f$.  The value $v_0(f)$ is the {\em non-hitting index} of $f$.
\end{definition}
Equivalently, $v_i(f)$ is the number of pairs $(a,b) \in \F_q^2$ such that the equation $f(x) = ax+b$ has precisely $i$ solutions in $\F_q$, i.e. the number of non-vertical lines which intersect the graph of $f$ in precisely $i$ points.  

\begin{example}[\cite{LiPot}]
Let $f=ax+b$ with $a,b\in\F_q$.  
We have $v_q(f)=1$ since the line $y=ax+b$ contains precisely the $q$ points of the graph of $f$. No other line contains more than one point of the graph, so $v_i(f)=0$ for $i=2,3,\dotsc,q-1$. We have $v_0(f)=q-1$, since the $q-1$ lines of the form $y=ax+c$ with $c\neq b$ do not intersect the graph of $f$. Finally, each remaining line intersects the graph of $f$ in precisely one point, hence $v_1(f)=q(q-1)$ (not counting vertical lines).  
\end{example}

We may also express these concepts in terms of projective geometry.  We associate $f$ with a subset $S_f$ of $\pg(2,q)$, the projective plane over $\F_q$, by setting
\begin{align*}
    S_f=\{(x,f(x),1):x\in\F_q\}\cup\{(0,1,0)\}.
\end{align*}
The set $S_f$ consists of $q+1$ points.  The point $(0,1,0)$ is the ``point at infinity'' that lies on all ``vertical'' lines of the form $ax+bz=0$ with $a,b\in\F_q$.  Note that if $f=\sum_{i=0}^n a_ix^i$ with $a_i\in \F_q$ and $a_n\neq 0$ with $n\geq 2$, then $(0,1,0)$ satisfies the homogeneous polynomial equation $yz^{n-1}=\sum_{i=0}^na_ix^iz^{n-i}$; i.e. the points of $S_f$ are precisely those points of $\pg(2,q)$ that satisfy this equation.  

For $0\leq i \leq q$, $v_i(f)$ is the number of lines of $\pg(2,q)$ of the form $y+ax+bz=0$ with $a,b\in\F_q$ that intersect $S_f$ in $i$ distinct points of $\pg(2,q)$.  We note that this definition ignores those lines of $\pg(2,q)$ that pass through $(0,1,0)$; these lines each intersect $S_f$ in two points, with the exception of the line at infinity $z=0$, which contains a single point of $f$. Hence, we do not lose information by ignoring these lines.

Connecting a polynomial $f$ with the set of points $S_f\subset\pg(2,q)$ allows geometric properties of $S_f$ to be brought to bear in the study of the behaviour of $f$.  To study such sets, the notion of intersection distribution can be extended to arbitrary sets $D$ of $q+1$ points in $\pg(2,q)$.  We first introduce some useful terminology.
\begin{definition}
Let $D\subset \pg(2,q)$.  A line of $\pg(2,q)$ that contains no points of $D$ is an {\em external line}. A line that intersects $D$ in a single point of $\pg(2,q)$ is a {\em unisecant}, a line that contains two distinct points of $D$ is a {\em bisecant}, and for $0\leq i \leq q+1$, a line that contains $i$ distinct points of $D$ is an {$i$-secant}.
\end{definition}

\begin{definition}[\cite{LiPot}]
Let $D\subset \pg(2,q)$ with $|D|=q+1$.  For $0\leq i \leq q+1$, let $u_i(D)$ denote the number of $i$-secants of $D$.   The sequence $(u_i(D))_{i=0}^{q+1}$ is the {\em intersection distribution} of $D$.  The value $u_0(D)$ is the {\em non-hitting index} of $D$.  The largest $i$ for which $u_i(D)$ is non-zero is the {\em degree} of $D$.
\end{definition}
There is a slight discrepancy between the intersection distribution of the set $S_f$ and the intersection distribution of the polynomial $f$, but conversion is straightforward:
\begin{theorem}[\cite{LiPot}]\label{thm:v_itou_i}
Let $f$ be a polynomial over $\F_q$ and let $S_f$ be the associated $(q+1)$-set in $\pg(2,q)$. Then
\begin{align*}
v_0(f)&=u_0(S_f),\\
v_1(f)&=u_1(S_f)-1,\\
v_2(f)&=u_2(S_f)-q,\\
v_i(f)&=u_i(S_f)\quad\quad (3 \leq i \leq q),\\
u_{q+1}(S_f)&=0.
\end{align*}
\end{theorem}

In \cite{LiPot}, Li and Pott establish the correspondence between the algebraic and geometric viewpoints, and determine the intersection distributions of various polynomials (chiefly monomials), mainly via calculation of the ``multiplicity distribution''. They give upper and lower bounds for the largest and smallest non-hitting index for $(q+1)$-sets in $\pg(2,q)$, and provide a characterisation/partial characterisation for sets with non-hitting index near the lower/upper bound respectively.  They introduce the notion of projective equivalence of polynomials, with intersection distribution being an invariant under the equivalence.  

In \cite{KyuLiPot}, Kyureghyan, Li and Pott describe the intersection distributions for cubic polynomials;  Li and Xiong do likewise for quartic polynomials in \cite{LiXiong}.  In \cite{KyuLiPot}, the multiplicity distribution is again the key tool, whereas in \cite{LiXiong}, a direct analysis is performed of the pattern of roots of quartic polynomials.  The question of which polynomials are projectively equivalent to $x^3$ and $x^4$ respectively is posed, with some conjectures made. In \cite{LiLiQu},  Li, Li and Qu resolve Conjecture 3.2(1) of \cite{KyuLiPot}, while in \cite{DinZie}, Ding and Zieve provide an alternative short proof of the second half of this conjecture.  

All of \cite{KyuLiPot, LiPot, LiXiong} contain applications of these concepts to other combinatorial problems: to Kakeya sets in \cite{KyuLiPot} and \cite{LiPot}, and to Steiner triple systems in \cite{KyuLiPot} and \cite{LiXiong}. \.{I}rima\u{g}z{\i}
and \"{O}zbudak present applications to Costas arrays and optical orthogonal codes in \cite{IriOzb}.

In this paper, we develop the ideas of intersection distribution and non-hitting index of Li and Pott in new directions.  On the geometric side, in Section \ref{sec:projequiv} we establish new results concerning the relationship between the intersection distribution and the projective equivalence of polynomials.  We highlight the role of the nuclei of the set, and provide a characterisation of polynomials $f$ for which $S_f$ possesses more than one internal nucleus, in terms of permutation polynomials.  In Section \ref{sec:algebraic}, we consider the degree of $S_f$ (the index of the largest non-zero entry in the intersection distribution of  $f$), focusing on the case when $f$ is monomial. We establish new results characterising/partially characterising the intersection of $S_f$ with classes of affine lines, allowing us to bound or precisely evaluate the degree in many cases.  In Section \ref{sec:specificvalues}, we provide a geometric and algebraic characterisation of several natural families of constructions and their intersection distributions; some examples from \cite{LiPot} are special cases of these. We extend what is known about the non-hitting spectrum, determining its next smallest possible values and showing that sets attaining these are not characterised by their non-hitting index. 
We resolve some of the Open Problems posed by Li and Pott in \cite{LiPot}, and present two short direct treatments of the intersection distribution of cubic polyomials, first established in \cite{KyuLiPot}. We conclude with some new Open Problems.

\subsection{Geometric background}\label{subsec:background}
Further detail on the concepts mentioned here, as well as proofs of the stated results, can be found in \cite{lracasse} or \cite{Hirschfeld}.
The {\em points} of $\pg(2,q)$ are represented using {\em homogeneous coordinates:} triples $(x,y,z)$ with $x,y,z\in\F_q$ not all zero, under the understanding that for any non-zero $\lambda\in \F_q$, $(x,y,z)$ and $(\lambda x, \lambda y, \lambda z)$ represent the same point.  {\em Lines} are sets of points satisfying equations of the form $ax+by+cz=0$ for $a,b,c\in\F_q$ not all zero.  A set of points is {\em collinear} if they all lie on a common line.  

A {\em collineation} $\pi$ of $\pg(2,q)$ is a bijection from the set of points of $\pg(2,q)$ to itself that preserves incidence, i.e. if $S$ is a set of collinear points, then so is $S^\pi$.  The group of collineations of $\pg(2,q)$ is denoted by $\pgammal(3,q)$.  Two sets of points $S$ and $S^\prime$ are said to be {\em projectively equivalent} if $S^\prime=S^\pi$ for some $\pi \in \pgammal(3,q)$.

An important class of collineations are the {\em homographies:} maps that arise by treating points as column vectors and multiplying them on the left by an invertible $3\times 3$ matrix over $\F_q$.  The group of homographies of $\pg(2,q)$ is denoted by $\pgl(3,q)$.  Given any ordered {\em quadrangle} (set of four points, no three of which are collinear), there is a unique homography that maps those points onto the points of the {\em fundamental quadrangle} $\{(1,0,0),(0,1,0),(0,0,1),(1,1,1)\}$.

Another class of collineations are {\em automorphic collineations:} these are maps that send a point $(x,y,z)$ to $(x^\sigma,y^\sigma,z^\sigma)$ where 
$\sigma$ is an automorphism of $\F_q$.  Automorphic collineations fix all the points of the fundamental quadrangle.  The {\em Fundamental Theorem of Field Planes} states that every collineation can be expressed as a product of a homography and an automorphic collineation.

We now introduce some terminology that will be used throughout this section.
\begin{definition}
Let $D\subset \pg(2,q)$ be a $(q+1)$-set.  A point $N\in \pg(2,q)\setminus{D}$ with the property that every line through $N$ is a unisecant of $D$ is called an (external) {\em nucleus} of $D$.  A point $I\in D$ that lies on one unisecant and $q$ bisecants of $D$ is an {\em internal nucleus} of $D$.
\end{definition}
Li and Pott observe that if $D$ is the set of points of a line of $\pg(2,q)$, then $D$ does not have an internal nucleus \cite{LiPot}.  However, the situation is different for sets of the form $S_f$ for some $f\in\F_q[x]$:
\begin{theorem}[\cite{LiPot}]\label{thm:sfinternalnucleus}
 Let $f\in\F_q[x]$.  Then the point $(0,1,0)$ is an internal nucleus of $S_f$.
Conversely, for any $(q+1)$-set $D\subset\pg(2,q)$ that has an internal nucleus, there exists a polynomial $f\in\F_q[x]$ of degree at most $q-1$ for which $D$ is projectively equivalent to $S_f$.
\end{theorem}
\begin{proof}
Any line of $\pg(2,q)$ through $(0,1,0)$ has an equation of the form $ax+bz=0$ for some $a,b \in \F_q$ not both zero.   Suppose $f\in\F_q[x]$.  The line   $z=0$ meets $S_f$ in the single point $(0,1,0)$.  A line $ax+bz=0$ with $a\neq 0$ meets $S_f$ in $(0,1,0)$ and the additional point $(-b/a,f(-b/a),1)$.  Hence, $(0,1,0)$ is an internal nucleus.

Now suppose $D$ is a $(q+1)$-set with an internal nucleus $I$.  There is a unique unisecant to $D$ that passes through $I$; let $P\neq I$ be any other point of this line.  By applying a suitable homography $\theta$ if necessary, we can let the coordinates of $I$ be $(0,1,0)$ and $P$ be $(1,0,0)$.  Then $D^\theta$ has $(0,1,0)$ as an internal nucleus and $z=0$ as the unisecant through this internal nucleus.

Let the remaining points of $D^\theta$ be $P_1,\dotsc,P_q$. By construction, these points have non-zero $z$-coordinate, so their coordinates can be written in the form  $(x_i,y_i,1)$ for $i\in\{1,\dotsc,q\}$.  As $(0,1,0)$ is an internal nucleus, we have $x_i\neq x_j$ for $i\neq j$.  Thus, by polynomial interpolation there is a unique polynomial $f\in\F_q[x]$ of degree at most $q-1$ with $y_i=f(x_i)$.
\end{proof}

One of the longest-studied structures in finite geometry is the {\em $k$-arc}, a set of points of $\pg(2,q)$ of which no three are collinear.  When $q$ is odd, the largest possible $k$-arc has size $q+1$ \cite{Bose}.  The intersection distribution for a $(q+1)$-arc can be completely determined by straightforward counting.
\begin{example}\label{ex:oddconic}
Let $\mathcal{K}$ be a $(q+1)$-arc in $\pg(2,q)$.  By definition, $u_i(\mathcal{K})=0$ for $i\geq 3$.  Consider a point $P\in \mathcal{K}$.  Any pair of points of $\mathcal{K}$ lie on a unique bisecant of $\mathcal{K}$, hence $u_2(\mathcal{K})=\binom{q+1}{2}=(q^2+q)/2$.  Every point $P$ of $\mathcal{K}$ is an internal nucleus: it lies on $q$ bisecants (one through each remaining point of $\mathcal{K}$) and the remaining line through $P$ is a unisecant. We thus have $u_1(\mathcal{K})=q+1$.  As $\pg(2,q)$ has $q^2+q+1$ lines in total, we conclude that
\[
    u_0(\mathcal{K})=q^2+q+1-\frac{q^2+q}{2}-q-1=\frac{q^2-q}{2}=\frac{q(q-1)}{2}.
\]
Thus, we have determined the intersection distribution (and non-hitting index) of $\mathcal{K}$.  \end{example}  

 A celebrated result of Segre shows that every $(q+1)$-arc consists of the points of a non-singular conic in $\pg(2,q)$ (a set of points that satisfy $F=0$ where $F\in \F_q[x,y,z]$ is an absolutely irreducible homogeneous quadratic polynomial \cite{Segre55}). An example
 of an irreducible conic is the set $S_f$ for any polynomial $f\in\F_q[x]$ of degree $2$.   It is known that for any $q$ there is a single orbit of non-singular conics under $\pgammal(3,q)$, i.e. any two $(q+1)$-arcs in $\pg(2,q)$ with $q$ odd are projectively equivalent. In particular, any $(q+1)$-arc in $\pg(2,q)$ with $q$ odd is projectively equivalent to $S_f$ where $f=x^2$ \cite{LiPot}.

When $q$ is odd, a $(q+1)$-arc in $\pg(2,q)$ does not have any external nuclei \cite{Hirschfeld}.  When $q$ is even, however, every $(q+1)$-arc $\mathcal{K}$ has a unique external nucleus $N$ \cite{Hirschfeld}.  It follows that $\mathcal{K}\cup{N}$ is a $(q+2)$-arc; it is known that this is the largest possible size of a $k$-arc in $\pg(2,q)$ with $q$ even.  A $(q+1)$-arc is frequently called an {\em oval}, and a $(q+2)$-arc is known as a {\em hyperoval.}\footnote{Caution: Hirschfeld \cite{Hirschfeld} uses the term {\em oval} to refer to the $(q+2)$-arc.} 

Let $I$ be a point of a $(q+1)$-arc $\mathcal{K}$.  Then $I$ is an internal nucleus.  By  Theorem~\ref{thm:sfinternalnucleus}, $\mathcal{K}$ is projectively equivalent to a $(q+1)$-arc that contains the points $(0,0,1)$ and $(1,1,1)$ and has internal nucleus $(0,1,0)$ and nucleus $(1,0,0)$; the remaining points of this arc have coordinates of the form $(x,f(x),1)$ for some $f\in\F_q[x]$ of degree $n$ with $2\leq n\leq q-1$. A polynomial arising from an oval in this manner is known as an {\em o-polynomial} \cite{segre57}.  Unlike in the case of odd $q$, for $q$ even it is possible to have projectively inequivalent ovals, and many constructions of ovals/o-polynomials are known.  The problem of classifying ovals up to projective equivalence is still open.   

A well-studied generalisation of $k$-arcs is the notion of a $(k;n)$-arc: a set of $k$ points of $\pg(2,q)$ that has at least one $n$-secant but no $i$-secant for $i\geq n$ (here we require $n\geq 2)$.  Typical questions studied are to identify the largest possible $(k;n)$-arcs that exist for particular choices of $n$ and $q$ (i.e. fixing $n$ and exploring the possibilities for $k$).  A $(q+1)$-set $D$ of degree $n$ is an example of a $(q+1;n)$-arc.  The research direction introduced by Li and Pott in \cite{LiPot} can be viewed as that of fixing $k$ to be $q+1$ and then exploring the possibilities for $n$ (as well as other properties of the associated intersection distribution).
The following well-established result, presented in Proposition~2.1 of \cite{LiPot} for the case $k=q+1$, shows that the values of the intersection distribution are not independent of each other.
\begin{theorem}[\cite{TalScaf}]\label{thm:uisums}
Let $D$ be a $(q+1)$-set in $\pg(2,q)$ of degree $n$. Then
\begin{itemize}
\item[(i)] $\sum_{i=0}^n u_i(D)=q^2+q+1$,
\item[(ii)] $\sum_{i=1}^n i u_i(D)=(q+1)^2$,
\item[(iii)] $\sum_{i=2}^n i(i-1) u_i(D)=q(q+1)$.
\end{itemize}
\end{theorem}
\begin{proof}
These results come from double counting the number of lines of $\pg(2,q)$, the number of incident point/line pairs $(P,\ell)$ with $P\in D$ and $\ell\in\pg(2,q)$, and the number of pairs of distinct points $P,Q\in D$ respectively.
\end{proof}
These results can be reformulated for the intersection degree of a polynomial by restricting to affine points and lines that do not pass through $(0,1,0)$ as follows:
\begin{theorem}\label{thm:visums}
Let $f\in\F_q[x]$.  Then
\begin{itemize}
    \item[(i)] $\sum_{i=0}^q v_i(f)=q^2$,
    \item[(ii)] $\sum_{i=1}^{q} i v_i(f)=q^2$,
    \item[(iii)] $\sum_{i=2}^q i(i-1) v_i(f)=q(q-1)$.
\end{itemize}
\end{theorem}

\section{Projective equivalence of polynomials}\label{sec:projequiv}
Open Problem 7 asked by Li and Pott \cite{LiPot} states:
``Explore the relation between the intersection distribution and the projective equivalence.  For instance, do there exist projectively inequivalent polynomials having the same intersection distribution?"

The existence of inequivalent ovals in $\pg(2,q)$ with $q$ even implies the answer to the latter question is yes, as the corresponding o-polynomials are projectively inequivalent by definition.  In this section, we explore projective equivalence more generally, and we give further examples of inequivalent polynomials that shed light on the relationship between the intersection distribution and the notion of projective equivalence.

We begin with a result characterising those projective transforms that, when applied to $S_f$ for any polynomial $f$, yield a set of the form $S_g$ for some polynomial $g$.

\begin{theorem}\label{thm:equiv}
Let $f=\sum_{i=0}^n a_i x^i$ be a polynomial over $\F_q$ with $a_n\neq 0 $. For an automorphism $\sigma$ of $\F_q$, we write $f^\sigma = \sum_{i=0}^n a_i^\sigma x^i$. Then for any automorphism $\sigma$ of $\F_q$, and any choice of elements $a,b,c,d,e\in \F_q$ with $a,e\neq 0$, we have that $f$ is projectively equivalent to the polynomial $g$ given by
\begin{align*}
g&=ef^\sigma(ax+b)+cx+d = e\sum_{i=0}^na_i^\sigma(ax+b)^i+cx+d.
\end{align*}
\end{theorem}
\begin{proof}
Let $f\in\F_q[x]$.  For any set $S_g$ where $g\in\F_q[x]$, the point $(0,1,0)$ is an internal nucleus and the line $z=0$ is the unique unisecant through this point.  
We observe that a collineation maps internal nuclei to internal nuclei and maps unisecants to unisecants; here we restrict our attention to those collineations that fix $(0,1,0)$ and fix the line $z=0$.  These collineations all send $S_f$ to $S_{g}$ for some polynomial $g$ that, by definition, is projectively equivalent to $f$.

We begin by observing that any automorphic collineation fixes $(0,1,0)$ and fixes the line $z=0$.  Suppose now that $P=(x,y,1)$ is an element of $S_g$, where $S_g$ is the image of $S_f$ under the automorphic collineation given by the automorphism $\sigma$.  Its preimage is $(x^{\sigma^{-1}},y^{\sigma^{-1}},1)$, and so we have $f(x^{\sigma^{-1}})=y^{\sigma^{-1}}$.  Applying $\sigma$ to both sides of this equation gives $y=f(x^{\sigma^{-1}})^{\sigma}$. For $f=\sum_{i=0}^na_ix^i$, we then have 
\[
    g(x)=\left(\sum_{i=0}^na_i(x^{\sigma^{-1}})^i\right)^\sigma =\sum_{i=0}^na_i^\sigma x^i = f^\sigma(x),
\]
as $\sigma$ is an automorphism.

We now consider homographies that fix $(0,1,0)$ and the line $z=0$.  Consider the homography $\pi$ given by a non-singular $3\times 3$ matrix $M$ with entries from $\F_q$.
If $\pi$ fixes the line $z=0$, then $M_{3,1}=M_{3,2}=0$; since $M$ is non-singular we then have $M_{3,3}\neq 0$.  As we obtain $\pi$ from any non-zero scalar multiple of $M$, without loss of generality we can set $M_{3,3}=1$. 

If $\pi$ fixes the point $(0,1,0)$, then $M_{1,2}=M_{3,2}=0$, and we have $M_{1,1},M_{2,2}\neq 0$ as $M$ is non-singular.  By setting $a=M_{1,1}^{-1}$, $b=-aM_{1,3}$, $c=aM_{2,1}$, $d=M_{2,3}+bca^{-1}$ and $e=M_{2,2}$, we can write $M$ as
\begin{align*}
M=\begin{pmatrix}
a^{-1}&0&-ba^{-1}\\
a^{-1}c&e&d-bca^{-1}\\
0&0&1
\end{pmatrix}.
\end{align*}

Let $P=(x,y,1)\in S_g$.  
We observe that $P^{\pi^{-1}}\in S_f$, and we have
\begin{align*}
P^{\pi^{-1}}&=\begin{pmatrix}
a&0&b\\
-ce^{-1}&e^{-1}&-e^{-1}d\\
0&0&1
\end{pmatrix}
\begin{pmatrix}
x\\y\\1
\end{pmatrix}=\begin{pmatrix}
ax+b\\
-ce^{-1}x+e^{-1}y-e^{-1}d\\
1
\end{pmatrix}
\end{align*}
As this is a point of $S_f$, we have
$e^{-1}(-cx+y-d)=f(ax+b)$, which rearranges to give $y=ef(ax+b)+cx+d$.
The result then follows from the Fundamental Theorem of Field Planes.
\end{proof}
\begin{corollary}\label{cor:oneintnuc}
Suppose $f=\sum_{i=0}^n a_i x^i$ is a polynomial over $\F_q$ that has only one internal nucleus.  Then any polynomial that is projectively equivalent to $f$ has the form  
\[
g=e\sum_{i=0}^na_i^\sigma(ax+b)^i+cx+d
\]
for some automorphism $\sigma$ of $\F_q$, and some $a,b,c,d,e\in \F_q$ with $a,e\neq 0$.
\end{corollary}
\begin{example}
Let $f=x^3\in\F_q[x]$ with $q$ odd.  Then every point $(\theta,\theta^3,1)$ with $\theta\neq 0$ is collinear with $(0,0,1)$ and $(-\theta, -\theta^3,1)$, and hence is not an internal nucleus.  This is thus an example of a polynomial $f$ for which $(0,1,0)$ is the only internal nucleus of $S_f$, and hence Corollary~\ref{cor:oneintnuc} completely determines the equivalence class of polynomials that are projectively equivalent to $f$.
\end{example}

It is interesting to compare the transforms described in Theorem~\ref{thm:equiv} to those arising from other notions of equivalence of polynomials.  Of particular relevance is the case of {\em Extended Affine (EA)-equivalence.}  
\begin{definition}
    Let $p$ be a prime.  Functions  $F,G\colon\F_p^m\rightarrow \F_p^m$ are said to be {\em EA-equivalent} if there are invertible affine transforms $A$ and $B$ of $\F_p^m$, and an affine transform $C$ of $\F_p^m$ such that $G(\mathbf{x})=(A\circ F\circ B)(\mathbf{x})+C(\mathbf{x})$.
\end{definition}
If we treat elements of $\F_{p^m}$ as elements of $\F_p^m$ by specifying a basis, then a polynomial in $\F_{p^m}[x]$ induces a vectorial $p$-ary function from $\F_p^m$ to $\F_p^m$, and we can view two such polynomials as being EA-equivalent if the corresponding $p$-ary vectorial functions are EA-equivalent.  We see immediately that the projective equivalences described in Theorem~\ref{thm:equiv} are examples of EA equivalence due to the fact that multiplication by an element of $\F_{p^m}$, or the application of an automorphism $\sigma$ of $\F_{p^m}$ are both linear operations over $\F_p$.  In the case of prime fields (i.e.\ $m=1$), every EA-equivalence is of this form, and is therefore also a projective equivalence.  However, for large $m$ it is readily seen that not all examples of EA-equivalence can arise this way, as demonstrated by the following example.
\begin{example}[Polynomials that are EA-equivalent but not projectively equivalent]
Let $\theta$ be a primitive element of $\F_9$ with minimal polynomial $x^2-x-1$.  Represent the element $x_1+\theta x_2$ for $x_1,x_2\in\F_3$ by $(x_1,x_2)^T$.  
The polynomial $f=x^2\in\F_9[x]$ induces the function
\[
    F\colon \begin{pmatrix}
        x_1 \\ x_2
    \end{pmatrix}\mapsto
    \begin{pmatrix}
        x_1^2+x_2^2 \\ 2x_1x_2+x_2^2
    \end{pmatrix}.
\]
Composing $F$ with the linear function given by the matrix $\begin{pmatrix}1&1\\0&2\end{pmatrix}$ gives the function
\[
    G\colon\begin{pmatrix}
        x_1 \\ x_2
    \end{pmatrix}\mapsto \begin{pmatrix}
        x_1+2 x_2^2+2x_1 x_2 \\ x_1x_2+2x_2^2
    \end{pmatrix},
\]
which is precisely the function arising from the polynomial $x^6\in \F_9[x]$.  
Thus, $x^2$ and $x^6$ are EA-equivalent over $\F_9$.  However, they are not projectively equivalent as they have different values of $v_0$ (see \cite{LiPot} Table 3.2).
\end{example}
The existence of EA-equivalences that are not projective equivalences arises from the fact that not all invertible linear transforms of $\F_p^m$ derive from multiplication by elements of $\F_{p^m}$ and/or application of an automorphism of $\F_{p^m}$, as well as the fact that automorphic collineations of $\pg(2,p^m)$ involve an automorphism applied to all three coordinates, whereas an EA-equivalence can be obtained by applying the corresponding $\F_p$-linear transformation to the $x$- and/or $y$-coordinates independently.

We conclude that the equivalence classes under EA-equivalence for polynomials $f$ where $S_f$ has a single internal nucleus are unions of projective equivalence classes.  If working over a prime field, they are precisely the projective equivalence classes.

When a polynomial has more than one internal nucleus, however, there are collineations additional to the ones described in Theorem~\ref{thm:equiv} that give rise to projectively equivalent polynomials.
\begin{example}[A projective equivalence not arising from Theorem~\ref{thm:equiv}] \label{ex:newprojequiv}
Consider $y=x^3$ over $\F_{2^m}$ with $m>2$.  Then the point $(0,0,1)$ is an internal nucleus.  If we swap the $y$- and $z$-coordinates, we swap this point with $(0,1,0)$ and swap the unisecant $y=0$ with the line at infinity $z=0$.  This can be achieved by the homography described by the matrix
\begin{align*}
\begin{pmatrix}
    1&0&0\\
    0&0&1\\
    0&1&0
\end{pmatrix}.
\end{align*}

Affine points other than $(0,0,1)$ have the form $(\theta, \theta^3,1)$ and thus get mapped to $(\theta,1,\theta^3)$.  Divide by $\theta^3$ (which is non-zero for all these points) to get the scalar multiple $(\theta^{-2},\theta^{-3},1)$.  Now apply the automorphism $\sigma:\theta\rightarrow \theta^{2^{m-1}}$ to each coordinate.  We get
\[
    \left(\sigma(\theta^{-2}),\sigma(\theta^{-3}),\sigma(1)\right)= (\theta^{-2^m},\theta^{-3\cdot2^{m-1}},1)=(\theta^{-1},\theta^{-2^{m-1}-1},1).
\]
The points of this form are precisely those that satisfy $y=x^{2^{m-1}+1}$. Hence $x^{2^{m-1}+1}$ is projectively equivalent to $x^3$.
\end{example}
We observe that to study such transformations, it suffices to consider polynomials $f$ for which $S_f$ has an internal nucleus at $(0,0,1)$.
\begin{lemma}\label{lem:originnucleus}
Let $f\in \F_q[x]$ be a polynomial for which $S_f$ has an internal nucelus at $(a,b,1)$.  Then $f$ is projectively equivalent to a polynomial $g$ with internal nucleus $(0,0,1)$ and for which the line $y=0$ is a unisecant.
\end{lemma}
\begin{proof}
Let $f$ be a polynomial for which $S_f$ has an internal nucleus at $(a,b,1)$.  There is a unique unisecant through this point, and its equation has the form $dx+y-(da+b)z=0$ for some $d\in\F_q$.
The following matrix gives rise to a homography that fixes $(0,1,0)$ and the line $z=0$, and which maps the point $(a,b,1)$ onto $(0,0,1)$ and the line $dx+y-(da+b)z=0$ onto the line $y=0$:
\begin{align*}
\begin{pmatrix}
    1&0&-a\\
    d&1&-da-b\\
    0&0&1
\end{pmatrix}
\end{align*}
This homography maps the points of $S_f$ onto the points of a $(q+1)$-set that has internal nuclei at $(0,1,0)$ and $(0,0,1)$, and for which the lines $z=0$ and $y=0$ are unisecants.  This set has the form $S_g$ for some polynomial $g\in\F_q[x]$ of degree at most $q-1$.
\end{proof}
We can characterise polynomials $f$ for which $S_f$ has an internal nucelus at $(0,0,1)$ in the following way.
\begin{lemma}\label{lem:xtimesperm}
    Let $f$ be a polynomial for which $S_f$ has an internal nucleus at $(0,0,1)$.  Then $f$ can be written in the form $f=xg$ for some permutation polynomial $g\in\F_q[x]$ with $g(0)=0$.
\end{lemma}
\begin{proof}
Let $f$ be a polynomial for which $S_f$ has an internal nucleus at $(0,0,1)$.  As $(0,0,1)$ satisfies $y=f(x)$, we have $f(0)=0$, so $f$ has the form $f=xg$ for some $g\in \F_q[x]$. We show that $g$ permutes the non-zero elements of $\F_q$.

Let $x,y \neq 0$. If $g(x)=g(y)$, then the points $(x,xg(x),1)$ and $(y,yg(y),1)=(y,yg(x),1)$ are elements of $S_f$. But
\begin{align*}
    \text{det}\begin{pmatrix}
        x&xg(x)&1\\
        y&yg(x)&1\\
        0&0&1
    \end{pmatrix}=0,
\end{align*}
so the points $(x,f(x),1)$, $(y,f(y),1)$ and $(0,0,1)$ are collinear. Since $(0,0,1)$ is an internal nucleus, we must have $x=y$.
Now if $g(0)=0$, then $g$ is a permutation polynomial; if $g(0)=c$, it can be written as $h+c(1-x^{q-1})$ for some permutation polynomial $h$ with $h(0)=0$.  But then
\[
    f(x)=xg(x)=x(h(x)+ c(1-x^{q-1}))= xh(x) + c(x-x^q) = xh(x),
\]
as $x^q=x$ for all $x\in\F_q$. Thus, we do not lose generality by supposing $g(0)=0$ and thus $g$ is a permutation polynomial.
\end{proof}
As permutation polynomials are well-studied objects with many known constructions, Lemma~\ref{lem:xtimesperm} could be used to obtain many examples of polynomials $f$ for which $S_f$ has more than one internal nucleus.  Here we apply it to obtain further information on the projective equivalence classes of such polynomials.
\begin{theorem}\label{thm:extraequiv}
Let $f$ be a polynomial of the form $f=xg$ for some permutation polynomial $g\in\F_q[x]$ with $g(0)=0$.   Then $f$ is projectively equivalent to the polynomial $xh$  where $h$ is the permutation polynomial with $h(0)=0$ that satisfies $h(x)=xg^{-1}(x^{-1})^{-1}$ for all $x\in\F_q^{*}$.
\end{theorem}
\begin{proof}
As in Example~\ref{ex:newprojequiv}, $f$
is projectively equivalent to a polynomial $f^\prime$ where the points of $S_{f^\prime}$ are obtained by swapping the $y$- and $z$-coordinates of the points of $S_f$, thereby swapping the internal nuclei $(0,0,1)$ and $(0,1,0)$.  Since $S_{f^\prime}$ has an internal nucleus at $(0,0,1)$, we can write $f^\prime=xh$ for some permutation polynomial $h$ satisfying $h(0)=0$.  The points of $S_{f^\prime}$ with non-zero $x$-coordinate have the form $(\theta,1,\theta g(\theta))$ with $\theta\in\F_q^*$.

As $g$ is a permutation polynomial with $g(0)=0$, it follows that  $\theta g(\theta)\neq 0$ when $\theta\neq 0.$ Thus, we have $(\theta,1,\theta g(\theta)) = (g(\theta)^{-1},\theta^{-1} g(\theta)^{-1},1)$.  Points of this form satisfy $y=xg^{-1}(x^{-1})^{-1}$, so for $x\in\F_q^*$ we have $h(x)=g^{-1}(x^{-1})^{-1}$.
\end{proof}
\begin{corollary}
Let $x^d$ be a permutation monomial in $\F_q[x]$ with inverse $x^e$.  Then $x^{d+1}$ is projectively equivalent to $x^{e+1}$.
\end{corollary}
\begin{example}
Consider $x^4\in\F_{11}[x].$  As $3\nmid 10$, we know $x^3$ is a permutation polynomial that fixes $0$.  The inverse of $x^3$ in $\F_{11}[x]$ is $x^7$.  Thus, by Theorem~\ref{thm:extraequiv}, we have that $x^4$ is projectively equivalent to $x^8$ over $\F_{11}$.  We observe that these functions are not EA-equivalent, as they have different degrees. 
\end{example}
This example shows that projective equivalence is not a refinement of EA-equivalence when considering polynomials whose point sets have more than one internal nucleus.  In general, the polynomials in the projective equivalence class of a polynomial $f$ will have the same number of internal nuclei as $f$, and can be obtained from $f$ by a sequence of those transformations described in Theorem~\ref{thm:equiv} and Theorem~\ref{thm:extraequiv}. We observe that o-polynomials are polynomials for which all $q+1$ points are internal nuclei.  This implies that the number of internal nuclei does not uniquely determine the projective equivalence class of a polynomial, even for polynomials with the same intersection distribution.

\subsection{Further examples of projectively inequivalent polynomials with the same intersection distribution}
The intersection distribution is a notable invariant of projective transformations: if two polynomials have different intersection distributions, they are necessarily projectively inequivalent.  We now turn our attention to some additional examples of projectively inequivalent polynomials that have the same intersection distribution, in order to cast light on the limitations of intersection distribution in characterising polynomials up to projective equivalence.

We begin by resolving the question posed by Li and Pott in Open Question 3 of \cite{LiPot} that asks which polynomials give rise to the $(q+1)$-sets of large non-hitting index described in their Examples 2.2 and 2.3, and we show that these sets provide an example of inequivalent polynomials having the same intersection distribution.
\subsubsection{Polynomials giving rise to the $(q+1)$-sets of Li and Pott's Examples 2.3 and 2.2}
We first consider the $(q+1)$-sets described in Example 2.3 of \cite{LiPot}.  These examples consist of a $(q+1)$-arc in $\pg(2,q)$ with $q$ odd, but with one point of the arc exchanged for another point that was not on the arc.  In odd characteristic a $(q+1)$-arc is a conic, and every conic is equivalent to that determined by the polynomial $y=x^2$.  Without loss of generality, we can choose to remove the point $(0,0,1)$.  If we replace it by another point on the line $x=0$, then the resulting $(q+1)$-set still has $(0,1,0)$ as an internal nucleus and hence is given by a polynomial of the form $y=f(x)$.  (We could replace it by a point on another line, but then we would have to perform a change of coordinates in order to determine the relevant polynomial, and the resulting $(q+1)$-set would be projectively equivalent to one of this type anyway.) We note that the tangent to the conic at $(0,0,1)$ is the line $y=0$.

If we replace the point $(0,0,1)$ with $(0,c,1)$ for some $c \in \F_q$, then the resulting $(q+1)$-set arises from the polynomial 
\begin{align*}
y=x^2+c(1-x^{q-1}).
\end{align*}
Notice that there is a unique polynomial of degree at most $q-1$ that is satisfied by the $q$ points other than $(0,1,0)$, and it is easily seen that the above polynomial works, as $x^{q-1}$ is $1$ for $x\neq 0$ and $0$ when $x=0$.

When $-c$ is a square, the new point is an external point of the conic: it lies on the tangents to the conic at the points $(\pm\sqrt{-c},-c,1)$, and so this set corresponds to Example 2.3(1).  When $-c$ is a non-square, it is an internal point (it lies on no tangents of the conic) and the set is the second type described in Example 2.3(2).  These are depicted in Figure~\ref{fig:oddqfirstexamples}.

We note in passing that when $-c$ is a non-square, then the $(q+1)$-set has only one internal nucleus, but when $-c$ is a square it has three internal nuclei, at $(0,1,0)$, $(\sqrt{-c},-c,1)$ and $(-\sqrt{-c},-c-1)$.  By applying  transforms that swap $(\sqrt{-c},-c,1)$ with $(0,1,0)$, we can also map this $(q+1)$-set onto one described by the polynomial $x^{q-2}+d(1-x^{q-1})$ for some non-zero $d$.
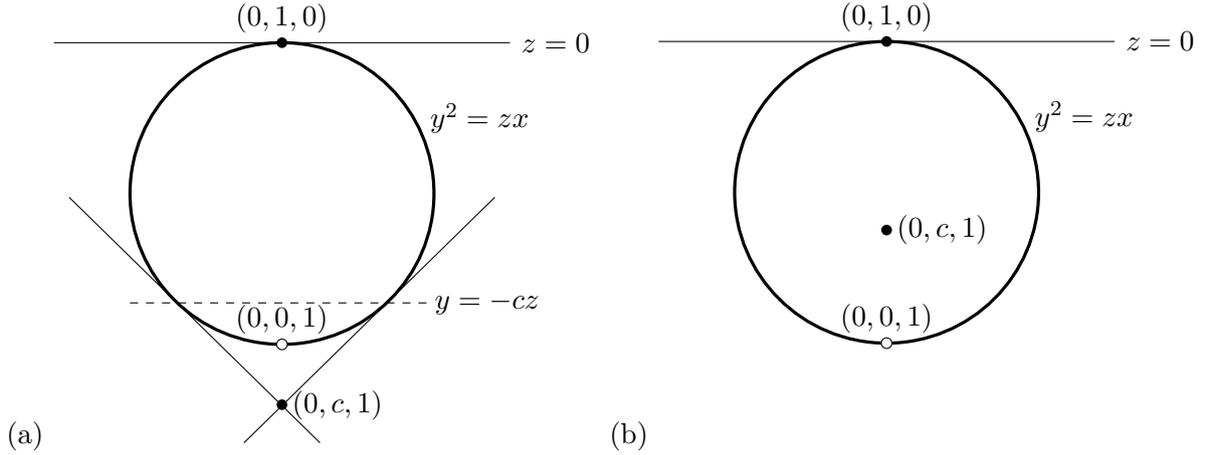
\begin{figure}
\begin{center}(a)
\begin{tikzpicture}
  \draw [very thick] (0,0)  circle  (2);

  \draw[-] (-3,2) -- (3,2);

  \fill (0,2)  circle (2pt) node[above] {$(0,1,0)$};
\draw [black, fill=white] (0,-2) circle (2pt) node[above] {$(0,0,1)$};
  \fill (0,-2.8) circle (2pt) node[right] {$(0,c,1)$};  

   \draw[-] (-.5,-3.305) -- (2.8,-.05);
   \draw[-] (.5,-3.305) -- (-2.8,-.05);
   \draw[dashed](-2,-1.45)--(2,-1.45);
   \node at (2.7,-1.45) {$y=-cz$};
   \node at (2.6,1) {$y^2=zx$};
   \node at (3.6,2) {$z=0$};
\end{tikzpicture}
(b)
\begin{tikzpicture}
  \draw [very thick] (0,0)  circle  (2);

  \draw[-] (-3,2) -- (3,2);

  \fill (0,2)  circle (2pt) node[above] {$(0,1,0)$};
\draw [black, fill=white] (0,-2) circle (2pt) node[above] {$(0,0,1)$};
  \fill (0,-.5) circle (2pt) node[right] {$(0,c,1)$};

   \node at (2.6,1) {$y^2=zx$};
   \node at (3.6,2) {$z=0$};
   \node at (0,-3.2) {};
\end{tikzpicture}
\end{center}
\caption{$(q+1)$-sets from (a) Example 2.3(1) and (b) Example 2.3(2) of \cite{LiPot}}\label{fig:oddqfirstexamples}
\end{figure}

We now turn our attention to the first type of $(q+1)$-set described in Example 2.3(2): a conic with a point replaced by an external point that lies on the tangent to the conic at the deleted point.  Here we choose our conic differently in order to ensure the resulting $(q+1)$-set has the line $z=0$ as a unisecant containing the internal nucleus $(0,1,0)$ without having to perform a subsequent change of coordinates.  
Start with the conic of equation $z^2=xy$.  The line $z=0$ intersects it in the points $(0,1,0)$ and $(1,0,0)$.  We note that the tangents to this conic at these points intersect in the point $(0,0,1)$.

We remove the point $(1,0,0)$ and replace it by the point $(0,0,1)$.  Observe that a point with coordinates $(\theta,\theta^{-1},1)$ for some $\theta\neq 0$ satisfies the equation $z^2=xy$, thus the remaining points of the conic all have this form.  From this, we see that the polynomial $f$ giving rise to this set is simply $y=x^{q-2}$, since this is satisfied by the remaining points of the conic, and also by the replacement point $(0,0,1)$.

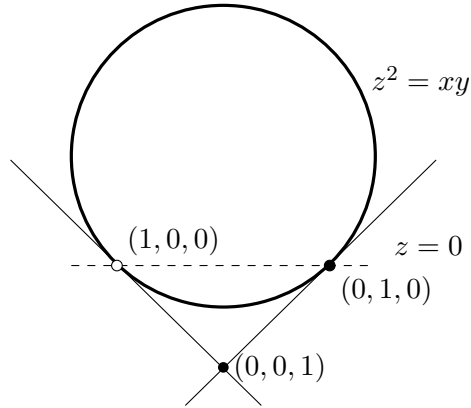
\begin{figure}
\begin{center}
    \begin{tikzpicture}
  \draw [very thick] (0,0)  circle  (2);
  \node at (2.7,-1.45) [above]{$z=0$};
\draw [fill] (1.4,-1.45) circle (2pt) node[below right] {$(0,1,0)$};

  \fill (0,-2.8) circle (2pt) node[right] {$(0,0,1)$};  

   \draw[-] (-.5,-3.305) -- (2.8,-.05);
   \draw[-] (.5,-3.305) -- (-2.8,-.05);
   \draw[dashed](-2,-1.45)--(2,-1.45);
   
   \node at (2.6,1) {$z^2=xy$};
   \draw [fill=white] (-1.4,-1.45) circle (2pt) node[above right] {$(1,0,0)$};
\end{tikzpicture}
\end{center}
    \caption{The first $(q+1)$-set of Example 2.3(2) of \cite{LiPot}}\label{fig:oddqsecondexample}
\end{figure}

This gives an example of two inequivalent polynomials with the same intersection distribution over $\F_q$ with $q$ odd with $q\geq5$: $y=x^2+c(1-x^{q-1})$ where $-c$ is a non-square, and $y=x^{q-2}$.  Both corresponding $(q+1)$-sets contain a maximal $q$-arc that lies in a unique conic, but in one case the remaining point is an internal point of the conic and in the other case it is an external point of the conic, hence the two sets are projectively inequivalent.

We now consider the $(q+1)$-set of Example 2.2, in which $q$ is taken to be even.  For case (1) of this example, we have the same situation as in case (1) of Example 3.3: a $(q+1)$-arc with a point removed and replaced by a point that is not on the tangent to the removed point.  There are some key differences to the characteristic 2 case: firstly, there exist projectively inequivalent $(q+1)$-arcs, which can lead to projectively inequivalent $(q+1)$-sets via this construction, and thus projectively inequivalent polynomials arising from it.

We start with the case where the $(q+1)$-arc is a conic.  Taking the conic of equation $yz+x^2$ and replacing the point $(0,0,1)$ with a point $(0,c,1)$ gives $S_f$ for $f=x^2+c(1+x^{q-1})$.  We note that another difference of the characteristic 2 case is that every point other than the nucleus (which is $(1,0,0)$ for the given conic) lies on a single tangent to the $(q+1)$-arc, hence we have the same intersection distribution for each non-zero choice of $c$.

More generally, if we take a $(q+1)$-arc and (without loss of generality) choose coordinates such that the line $z=0$ is a tangent to the arc at the point $(1,0,0)$ and such that $(1,0,0)$ is its nucleus, then the remaining points are of the form $(x,f(x),1)$ for some o-polynomial $f$.  Thus for any o-polynomial $f$ and non-zero $c$, the polynomial $f+c(1+x^{q-1})$ corresponds to a set arising from this construction. 

Example 2.2 (2) does not have an internal nucleus, and hence does not arise from a polynomial.  To see this, we note that the point $Q$ lies on a single tangent of the original $(q+1)$-arc, so each of the remaining $q$ points of the arc lie on a $3$-secant through $Q$, hence no remaining point is an internal nucleus.

\subsubsection{Permutation Polynomials and their Inverses}
Permutation polynomials and their inverses constitute an interesting example of one way in which projectively inequivalent polynomials having the same intersection distribution can arise. Note that o-polynomials are a special case of this phenomenon.

\begin{lemma}
If $f\in\F_q[x]$ is a permutation polynomial, then $(1,0,0)$ is a nucleus of $S_f$.
\end{lemma}
\begin{proof}
The lines through $(1,0,0)$ are precisely the $q$ lines with equations $y+cz=0$ for some $c \in \F_q$ and the line with equation $z=0$. The latter intersects $S_f$ only in the point at infinity $(0,1,0)$. The other $q$ lines correspond to the affine ``horizontal'' lines with equations $y=c$ for some $c\in\F_q$. If $f$ is a permutation polynomial, then by definition there is a unique point of $S_f$ on each such line.
\end{proof}

Suppose we take the set $S_f$ and replace the point $(0,1,0)$ by the point $(1,0,0)$ to obtain a new $(q+1)$-set $S^\prime$.  As we have not changed any of the affine points, we see that $S^\prime$ has the same intersection distribution as $S_f$.  Now, the points of $S^\prime\setminus (1,0,0)$ were those that satisfy $y=f(x)$, hence they also satisfy $x=f^{-1}(y)$.  Applying the projectivity determined by the matrix
\begin{align*}
\begin{pmatrix}
    0&1&0\\
    1&0&0\\
    0&0&1
\end{pmatrix}
\end{align*}
swaps the $x$- and $y$-coordinates of points in $S^\prime$, and thus sends $S^\prime$ into the $(q+1)$-set determined by the polynomial $y=f^{-1}(x)$.  As this set is projectively equivalent to $S^\prime$, it has the same intersection distribution.  While in some cases it may also be projectively equivalent to $S_f$, in general this will not be the case, as shown by the following example.
\begin{example}\label{ex:nonequivinverse}
Consider the polynomial $y=x^3$ over $\F_q$ with $q\neq 1\pmod{3}$ and $q>9$.  This is a known example of a permutation polynomial. The points of $S_f$ are precisely the rational points of the cubic curve of the equation $z^2y=x^3$.  Any collineation maps a cubic curve into another cubic curve.  So, if there is a collineation that maps $S_f$ into $S^\prime$, as defined above, then the points of $S^\prime$ are those of a cubic curve.  But $S^\prime$ and $S_f$ have $q\geq 10$ points in common.  This contradicts the fact that by B\'{e}zout's Theorem, any two distinct cubics without a common component intersect in at most 9 points.  Hence, we conclude that $S_f$ and $S^\prime$ are not projectively equivalent, and therefore neither are $f$ and $f^{-1}$.
\end{example}
Example~\ref{ex:nonequivinverse} provides an interesting contrast between projective equivalence and the notion of {\em CCZ-equivalence} introduced by Carlet, Charpin and Zinoviev in \cite{ccz}.
\begin{definition}
Let $p$ be a prime.  Functions $F,G\colon\F_p^m\rightarrow\F_p^m$ are {\em CCZ-equivalent} if there is an invertible affine transform of $\F_p^m\times\F_p^m$ that maps the set $\{(\mathbf{x},F(\mathbf{x})):\mathbf{x}\in\F_p^m\}$ to the set $\{(\mathbf{x},G(\mathbf{x})):\mathbf{x}\in\F_p^m\}$.
\end{definition}
EA-equivalence is a special case of CCZ-equivalence, so those projective equivalences that coincide with EA-equivalences are also CCZ-equivalences; likewise those EA-equivalences that are not projective equivalences also give examples of CCZ-equivalences that are not projective equivalences. 

Let $f\in\F_{p^m}[x]$ be a permutation polynomial. In general, $f$ is not EA-equivalent to its inverse polynomial $f^{-1}$ as the corresponding vectorial function from $\F_p^m$ to $\F_p^m$ does not necessarily have the same degree.  However, $f$ and $f^{-1}$ are CCZ-equivalent since the map that sends $(x,y)\in\F_{p^m}^2$ to $(y,x)\in\F_{p^m}^2$ can be viewed as an invertible affine transform of $\operatorname{AG}(2m,p)$ that maps the graph of $f$ to that of $f^{-1}$.  This affine transform can also be interpreted as the collineation of $\pg(2,p^m)$ described in Example~\ref{ex:nonequivinverse}.  This does not give a projective equivalence between $S_f$ and $S_{f^{-1}}$: when applied to $S_f$ it sends the point at infinity $(0,1,0)$ to $(1,0,0)$, which is not an element of $S_{f^{-1}}$.  

It would be interesting to determine whether the projective equivalence described in Theorem~\ref{thm:extraequiv} is also a CCZ-equivalence in general, as this would suffice to determine whether projective equivalence is a refinement of CCZ-equivalence.

\section{Algebraic viewpoint}\label{sec:algebraic}

From an algebraic viewpoint, determining the intersection distribution of a polynomial $f \in \F_q[x]$ is equivalent to counting the number of distinct roots over $\F_q$ of $f-ax-b$ (without multiplicity), as $(a,b)$ runs through $\F_q^2$.  When $f$ is a monomial $x^d$, this is a question about the roots of trinomials.

Relevant results occur in the polynomial literature.  In \cite{Blu}, for $q=p^s$, it is determined that $x^{{p^i}+1}+cx+d$ ($cd \neq 0$) has $n$ roots in $\F_q$, where $n \in \{0,1,2,p^{\mathrm{gcd}(i,s)}+1\}$, and the number of $c,d$ attaining each $n$ are found; this is equivalent to determining the non-zero entries of the intersection distribution of $f=x^{p^i+1}$ and the complete intersection distribution, respectively.   In \cite{Vil}, the number of possible roots of $x^{p^i}+cx+d$ ($c \neq 0$) are identified and the complete intersection distribution of $x^{p^i}$ is determined for some cases, an analysis which is completed in \cite{CouHen}. (The intersection distributions for both  are also calculated in \cite{LiPot}.)  In general however, the focus of the polynomial literature has been analysing the roots and factor structure for fixed $c$ and $d$, often motivated by applications.

In this section, we show that it is fruitful to analyse the polynomials corresponding to the intersection of $y=x^d$ with $y=ax+b$, by separately considering the cases $y=ax$, $y=b$ and $y=ax+b$ with $a,b \neq 0$.

\subsection{The degree of $S_f$}
\label{subsec:degree}

Recall that the degree of $S_f$ is the largest $i$ such that $v_i(f)$ is non-zero, i.e. the largest number of roots in $\F_q$ possible for the polynomial $f-ax-b$ as $(a,b)$ ranges over $\F_q^2$. It is clear that the degree of $S_f$ is a very natural concept when the problem is viewed algebraically.  This concept was defined in \cite{LiPot}, but was not a key focus in that paper nor its follow-up works \cite{KyuLiPot}, \cite{LiXiong}. In this section, we establish an understanding of the degree of $S_f$, focussing on the case when $f$ is a monomial $x^d$.  

We begin with some upper bounds.

\begin{lemma}\label{lem:degbound}
Let $d \geq 2$.  For a polynomial $f$ over $\F_q$ of degree $d$, $\mathrm{deg}(S_f) \leq d$.
\end{lemma}
\begin{proof}
It is well-known that the number of roots of a polynomial over a field is at most the degree of the polynomial; the polynomial $f-ax-b$ has degree $d$, and so has at most $d$ roots for any $(a,b) \in \F_q^2$.
\end{proof}

\begin{theorem}\label{thm:ubound}
Let $2 < d < q-1$ and $f=x^d$.  Then 
$$\mathrm{deg}(S_f) \leq \mathrm{min}(d,q-d+1).$$
Specifically, an upper bound for the number of points of intersection of $y=f(x)$ and $y=ax+b$ is given by
\begin{itemize}
\item[(i)] $q-d+1$ if $a \neq 0$ and $b=0$;
\item[(ii)] $q-d$ if $a,b \neq 0$;
\item[(iii)] $q-d-1$ if $a=0$ and $b \neq 0$.
\end{itemize}
\end{theorem}
\begin{proof}
Since $f=x^d$ has degree $d>2$, applying the above lemma shows that $\deg(S_{x^d}) \leq d$.  Next, we note that $x=0$ is a root of $x^d-ax-b$ when $b=0$, and not when $b \neq 0$.

Now, consider non-zero roots.  As $\alpha^{q-1}=1$ for all $\alpha \in \F_q^*$, the number of non-zero roots of $x^d-ax-b$ is the same as the number of non-zero roots of
$$x^{q-1-d}(x^d-ax-b)=1-ax^{q-d}-bx^{q-1-d}=-ax^{q-d}-bx^{q-d-1}+1$$
where $q-1-d>0$.  This polynomial has at most $q-d$ roots if $a \neq 0$, at most $q-d-1$ roots if $a=0$ and $b \neq 0$, and $0$ roots otherwise.  There is at most one further root ($x=0$).  The result now follows.
\end{proof}

We prove a useful corollary of Theorem \ref{thm:ubound}.

\begin{corollary}\label{cor:q-p^i}
Let $q=p^s$ ($s>1$) and let $f=x^{q-p^i}$ where $i\mid s$, $i<s$.  Then $\mathrm{deg}(S_f) \leq p^i$.
\end{corollary}
\begin{proof}
Let $d=q-p^i$ and $f=x^d$.  By Theorem \ref{thm:ubound}, an upper bound for $\mathrm{deg}(S_f)$ is given by $\mathrm{min}(p^s-p^i,p^i+1)$. In the special case $p=s=2$ and $i=1$, we have $p^s-p^i = 2 =p^i$. Otherwise, we have $\mathrm{min}(p^s-p^i,p^i+1)=p^i+1$. We show this is not attainable.

Consider $x^{q-p^i}-ax-b=0$, where $a,b \in \F_q$. Here $x=0$ is a root only when $b=0$. If $a=b=0$, it is the sole root, so assume $a,b$ are not both $0$.
 
Now consider non-zero roots. From above, $x^{q-p^i}-ax-b$ has the same non-zero roots over $\F_q$ as $-ax^{p^i}-bx^{p^i-1}+1$.  This polynomial has at most $p^i$ roots, with the possible maximum value of $p^i+1$ for $\mathrm{deg}(S_f)$ achievable only when $b=0$ and $a \neq 0$ (so that $0$ is also a root).  To achieve this upper bound, we would require the polynomial to have $p^i$ distinct non-zero roots in this case. The equation becomes $ax^{p^i}-1=0$, i.e. $x^{p^i}-a^{-1}=0$.  But $a^{-1}=r \in \F_q^*$ and so $r=r^q=(r^{p^{s-i}})^{p^i}$, i.e. $x^{p^i}-a^{-1}=(x-r^{p^{s-i}})^{p^i}$.  Hence, we do not get $p^i$ distinct roots in this case, so the possible maximum $p^i+1$ cannot be achieved.  
\end{proof}

Next, we separately analyse the three types of lines whose intersections contribute to the intersection distribution of $f=x^d$, and the resulting polynomials.

We first demonstrate that intersections of $y=x^d$ with lines $y=b$ (``horizontal" lines), and $y=ax$ (lines through the origin) can be described completely, using the cyclotomic structure of $\F_q$.

Recall the following properties of polynomials over finite fields (see for example \cite{LidNie}).

\begin{lemma}\label{lem:roots}
Let $q$ be a prime power.  Then for polynomials over $\F_q$,
\begin{itemize}
\item[(i)] the distinct roots in $\F_q$ of a polynomial $f \in \F_q[x]$  are precisely those of $\mathrm{gcd}(f, x^q-x)$ (respectively $\mathrm{gcd}(f, x^{q-1}-1)$ for the non-zero roots of $f$). The number of such roots is given by the degree of $\mathrm{gcd}(f, x^q-x)$ (respectively $\mathrm{gcd}(f, x^{q-1}-1)$).
\item[(ii)] For $1 \leq a,b \leq q-1$, $gcd(x^a-1,x^b-1)=x^{\mathrm{gcd}(a,b)}-1$.
\end{itemize}
\end{lemma}

We will also need the following definition:
\begin{definition}
Let $\alpha$ be a primitive element of $\F_q$ and let $e \mid q-1$. The cyclotomic classes of order $e$ in $\F_q$ are defined to be
$$ C_i^e=\alpha^i \langle \alpha^e \rangle$$
where $0 \leq i \leq e-1$.  The class $C_0^e$ is a multiplicative subgroup of $\F_q^*$.  Each cyclotomic class has cardinality $(q-1)/e$ and the $e$ cyclotomic classes of order $e$ partition $\F_q^*$.
\end{definition}

We can now describe the situation for the intersection of monomials with horizontal lines.

\begin{theorem}\label{thm:(d,q-1)}
Let $f=x^d$ and let $m=\mathrm{gcd}(d,q-1)$. Let $\alpha$ be a primitive element of $\F_q$.  Then 
\begin{itemize}
\item[(i)] $y=f(x)$ has $m$ points of intersection with each line of the form $y=\alpha^{id}$, where $0 \leq i \leq (q-1)/m$;
\item[(ii)] $y=f(x)$ has no points of intersection with any other horizontal line of the form $y=b$, $b \in \F_q$.
\end{itemize}
\end{theorem}
\begin{proof}
By Lemma \ref{lem:roots}, the distinct roots of $x^d-1$ over $\F_q$ are precisely the roots of
\[
    \mathrm{gcd}(x^d-1,x^{q-1}-1)=x^{\mathrm{gcd}(d,q-1)}-1=x^m-1.
\]
By definition, $m \mid d$ and $m \mid q-1$; let $q-1=um$ and $d=vm$ for integers $u,v$. The points of intersection of $y=f(x)$ and $y=1$ are the $m$ roots of $x^m-1$ in $\F_q^*$, i.e. the elements of the cyclotomic class $C_0^{(q-1)/m}=C_0^u$.

Recall that the cyclotomic classes of order $u$ partition the non-zero elements of $\F_q$.  Consider the elements of $C_i^u$ ($0 \leq i \leq u-1$).  For $\alpha^{ku+i} \in C_i^u$, we have
\[
    f(\alpha^{ku+i})=\alpha^{dku} \alpha^{di}=\alpha^{vmku}\alpha^{di}=\alpha^{di}=f(\alpha^i).
\]
So there are $u$ parallel affine lines $y=\alpha^{id}$ $(0 \leq i \leq u-1)$ which each intersect in precisely $m$ points with the graph of $f$. These lines partition all affine points on the graph except for the point $(0,0)$.
\end{proof}

\begin{example}
Let $q=13$, and take $\alpha=2$.
\begin{itemize}
\item[(i)] Let $d=4$, so $d \mid q-1$, $m=\mathrm{gcd}(d,q-1)=d$ and $f=x^4$.  The three cyclotomic classes of order $3$ in $\F_{13}$ are
\[
    C_0^3=\{1,5,8,12\},\;C_1^3=\{2,3,10,11\},\; C_2^3=\{4,6,7,9\}.
\]
The points $(x,f(x))$ with
\begin{itemize}
    \item[$\bullet$]{
        $x \in C_0^3$ are $\{(1,1), (5,1), (8,1), (12,1)\}$ (all lie on the line $y=1$),
    }
    \item[$\bullet$]{
        $x \in C_1^3$ are $\{(2,3), (3,3), (10,3), (11,3)\}$ (all lie on the line  $y=3$),
    }
    \item[$\bullet$]{
        $x \in C_2^3$ are $\{(4,9), (6,9), (7,9), (9,9)\}$ (all lie on the line $y=9$).
    }
\end{itemize}
Observe that the set of $y$-coordinates $\{1,3,9\}$ is precisely the set $C_0^{(q-1)/d}=C_0^4$ in $\F_{13}$.  
\item[(ii)] Let $d=8$; here $m=\mathrm{gcd}(d,q-1)=\mathrm{gcd}(8,12)=4$.  The distinct roots of $x^8-1$ in $\F_{13}$ are the distinct roots of $x^4-1$, namely the elements of $C_0^3$. So the points $(x,x^8)$ with
\begin{itemize}
    \item[$\bullet$]{
        $x \in C_0^3$ are $\{(1,1), (5,1), (8,1), (12,1)\}$ (line $y=1$),
    }
    \item[$\bullet$]{
        $x \in C_1^3$ are $\{(2,9), (3,9),(10,9),(11,9)\}$ (line $y=9$),
    }
    \item[$\bullet$]{
        $x \in C_2^3$ are $\{(4,3), (6,3), (7,3), (9,3)\}$ (line $y=3$).
    }
\end{itemize}
Again, the set of $y$-coordinates is $\{1,3,9\}$.
\end{itemize}
\end{example}

The next corollary gives sufficient conditons for the horizontal lines to achieve $\mathrm{deg}(S_f)$.

\begin{corollary}\label{cor:(d,q-1)}
Let $f=x^d$ and let $m=\mathrm{gcd}(d,q-1)$. If $m(m-1) \geq q-1$ then $\mathrm{deg}(S_f)=m$.
\end{corollary}
\begin{proof}
By Theorem \ref{thm:(d,q-1)}, the non-zero affine points of the graph of $f$ are partitioned into $(q-1)/m$ parallel lines (each intersecting the graph in $m$ points); any other affine line distinct from these can contain at most one point from each (and possibly the point $(0,0)$), giving a maximum number of $1+(q-1)/m$ intersection points. If $m(m-1) \geq q-1$, then $m \geq 1+(q-1)/m$; in this case $m$ is the maximum number of intersection points on any line and so $\mathrm{deg}(S_f)=m$.
\end{proof}

\begin{example}
Let $q=25$ and $d=16$; let $f=x^{16}$.  Theorem \ref{thm:ubound} shows $\mathrm{deg}(S_f) \leq 10$, but in fact $\mathrm{deg}(S_f)=8$ by Corollary \ref{cor:(d,q-1)}, since $m=\mathrm{gcd}(8,24)=8$ and $m(m-1)=56>24=q-1$.  Let $\alpha$ be a primitive element of $\F_{25}$.  The $24$ non-zero affine points of $S_f$ may be partitioned into three sets, each containing $8$ points (those with $x \in C_i^3$, $0 \leq i \leq 2$), corresponding to intersections with the lines $y=1$, $y=\alpha^8$ and $y=\alpha^{16}$ (i.e. all $y=c$ where $c \in C_0^3$).  For example, the points $(x,f(x))$ on $y=1$ are 
\[
    \{(1,1), (\alpha^3,1), (\alpha^6,1), (\alpha^9,1), (\alpha^{12},1), (\alpha^{15},1), (\alpha^{18},1), (\alpha^{21},1)\}.
\]
Any other line intersecting the graph of $f$ can have at most one point of intersection with each of these three lines, plus possibly $(0,0)$, i.e. at most $4\;(<8)$ points.
\end{example}

Similarly, we can describe the situation for lines through the origin.

\begin{theorem}\label{thm:(d-1,q-1)}
Let $f=x^d$.  Let $m=\mathrm{gcd}(d-1,q-1)$. Let $\alpha$ be a primitive element of $\F_q$. Then
\begin{itemize}
\item[(i)] $y=f(x)$ has $m+1$ points of intersection with each line of the form $y=\alpha^{(d-1)i}x$, where $0 \leq i \leq (q-1)/m$;
\item[(ii)] $y=f(x)$ has no points of intersection with any other line of the form $y=ax$, $a \in \F_q^*$.
\end{itemize}
\end{theorem}
\begin{proof}
By Lemma \ref{lem:roots} , the distinct roots in $\F_q$ of $x^d-x$ are precisely those of
\[
    \mathrm{gcd}(x^{d}-x,x^q-x)=x(\mathrm{gcd}(x^{d-1}-1,x^{q-1}-1))=x(x^{\mathrm{gcd}(d-1,q-1)}-1)=x(x^m-1).
\]
By definition, $m \mid d-1$ and $m \mid q-1$; let $q-1=um$ and $d-1=vm$ for integers $u,v$.  The points of intersection of $y=f(x)$ and $y=x$ are the point $(0,0)$ together with the $m$ roots of $x^m-1$ in $\F_q^*$, i.e. the elements of the cyclotomic class $C_0^{(q-1)/m}=C_0^u$.  Recall that the cyclotomic classes of order $u$ partition the non-zero elements of $\F_q$. Consider the elements of $C_i^u$ ($0 \leq i \leq u-1$). For $\alpha^{ku+i} \in C_i^u$, we have
\[
    f(\alpha^{ku+i})=\alpha^{dku} \alpha^{di}=\alpha^{(vm+1)ku} \alpha^{di}=\alpha^{ku}\alpha^{di}=\alpha^{ku+i}\alpha^{(d-1)i}.
\]
So the $m$ elements of $C_i^u$ together with $0$ are the roots of $x^d-\alpha^{(d-1)i} x$.  For each $0 \leq i \leq u-1$, this gives a set of $u$ lines $y=\alpha^{(d-1)i} x$, each of which intersects in $m+1$ points with the graph of $f$, and which have only the point $(0,0)$ in common. This accounts for all affine points of the graph.
\end{proof}

\begin{example}
Let $q=13$ and $\alpha=2$. Let $d=5$ and $f=x^5$. Then
\begin{itemize}
    \item{
        the points $\{(0,0),(1,1), (5,1), (8,1), (12,1)\}$ correspond to the roots of $f-x=x(x^4-1)$,
    }
    \item{
        the points $\{(0,0),(2,3), (3,3), (10,3), (11,3)\}$ correspond to roots of $f-3x=x(x^4-3)$,
    }
    \item{
        the points $\{(0,0),(4,9), (6,9), (7,9), (9,9)\}$ correspond to roots of $f-9x=x(x^4-9)$.
    }
\end{itemize}
\end{example}

\begin{corollary}\label{cor:(d-1,q-1)}
Let $f=x^d$ and let $m=\mathrm{gcd}(d-1,q-1)$. If $m(m+1) \geq q-1$ then $\mathrm{deg}(S_f)=m+1$.
\end{corollary}
\begin{proof}
Any line distinct from the $(q-1)/m$ lines (each containing $m+1$ points) described in Theorem \ref{thm:(d-1,q-1)} can contain at most one point from each, and if $m(m+1) \geq q-1$ then $m+1 \geq (q-1)/m$.
\end{proof}

\begin{example}
Let $q=31$, $d=26$ and $f=x^{26}$. Then
\[
    \mathrm{gcd}(d,q-1)=\mathrm{gcd}(26,30)=2\; \text{ and } \;\mathrm{gcd}(d-1,q-1)=\mathrm{gcd}(25,30)=5,
\]
so we use Corollary \ref{cor:(d-1,q-1)}.  Here, $5(5+1)=30=q-1$, so $\mathrm{deg}(S_f)=6$.

Let $\alpha$ be a primitive element of $\F_{31}$. The six lines of intersection, each containing six points of the graph (with $(0,0)$ in common to all lines) are $y=x$, $y=\alpha^{5} x$, $y=\alpha^{10}x$, $y=\alpha^{15}x$, $y=\alpha^{20}x$ and $y=\alpha^{25}x$.
\end{example}

A direct consequence of Theorems \ref{thm:(d,q-1)} and \ref{thm:(d-1,q-1)} is the following lower bound result.

\begin{corollary}\label{cor:lbound}
Let $2 \leq d \leq q-1$ and let $f=x^d$. Then
\begin{itemize}
\item[(i)] $\mathrm{deg}(S_f) \geq \mathrm{gcd}(d,q-1)$;
\item[(ii)] $\mathrm{deg}(S_f) \geq \mathrm{gcd}(d-1,q-1)+1$;
\end{itemize}
i.e. $\mathrm{deg}(S_f) \geq \mathrm{max}\big(\mathrm{gcd}(d,q-1), \mathrm{gcd}(d-1,q-1)+1\big)$.
\end{corollary}

By combining these results, we can evaluate $\mathrm{deg}(S_f)$ for several types of monomials $f$.  This includes as special cases all monomials featured in Table 3.1 of \cite{LiPot} that do not correspond to those of  \cite{CouHen} or \cite{Blu} (except $x^{q-2}$ when $q$ is even, considered separately below).

\begin{theorem}\label{thm:summarye|q-1}
Let $e$ be a proper divisor of $q-1$.
\begin{itemize}
\item[(i)] Let $f=x^{q-1}$; then $\mathrm{deg}(S_f)=q-1$.
\item[(ii)] Let $f=x^e$; then $\mathrm{deg}(S_f)=e$.
\item[(iii)] Let $f=x^{e+1}$; then $\mathrm{deg}(S_f)=e+1$.
\item[(iv)] Let $f=x^{q-e}$; then $\mathrm{deg}(S_f)=e+1$.
\end{itemize}
\end{theorem}
\begin{proof}
We show that the stated values give both an upper and lower bound for $\mathrm{deg}(S_f)$ in each case.  Case (i) is immediate from Lemma \ref{lem:degbound} and Theorem \ref{thm:(d,q-1)}.

Otherwise, by Theorem \ref{thm:ubound},  $\mathrm{deg}(S_{x^d}) \leq \mathrm{min}(d,q+1-d)$.
\begin{itemize}
    \item{
        In case (ii), $\mathrm{min}(e,q+1-e)=e$.
    }
    \item{
        In case (iii), $\mathrm{min}(e+1,q+1-(e+1))=e+1$.
    }
    \item{
        In case (iv), $\mathrm{min}(q-e,q+1-(q-e))=\mathrm{min}(q-e,e+1)=e+1$.
    }
\end{itemize}
In all cases, this follows from the fact that $e \leq (q-1)/2$, since $e$ is a proper divisor of $q-1$.

We next consider lower bounds.  For (ii) and (iii), the results are immediate from Corollary \ref{cor:lbound}:
\begin{itemize}
    \item{
        In case (ii), $\mathrm{gcd}(e,q-1)=e$.
    }
    \item{
        In case (iii), $\mathrm{gcd}((e+1)-1,q-1)+1=e+1$.
    }
    \item{
        In case (iv), since $e \mid q-1$, it follows that $e \mid q-1-e$.  Since also $\mathrm{gcd}(q-1,q-1-e) \mid e$, we have that $\mathrm{gcd}(q-1,q-1-e)=e$
    }
\end{itemize}
The result then follows from Corollary \ref{cor:lbound}(ii).
\end{proof}

\begin{proposition}
Let $q=2^s$ ($s>1$) and let $f=x^{q-2}$. Then $\mathrm{deg}(S_f)=2$. 
\end{proposition}
\begin{proof}
Clearly $\mathrm{deg}(S_f) \geq 2$ , while by Corollary \ref{cor:q-p^i}, $\mathrm{deg}(S_f) \leq 2$.
\end{proof}

\begin{remark}
The case $f=x^{q-2}$ may be set in a geometric context as follows.  Recall that a $(q+1)$-set with degree 2 is an oval, which is equivalent to a set given by an o-polynomial. The polynomial $x^{q-2}$ is a known example of an o-polynomial (see \cite{Hirschfeld}) and the corresponding oval consists of a conic with one point replaced by the nucleus.  The picture is essentially the same as Figure~\ref{fig:oddqsecondexample}; the difference between $q$ even and $q$ odd is that in the even case, the point $(0,0,1)$ is the nucleus of the conic, whereas in the odd case it lies on only two tangents, so the degree is $3$.
\end{remark}

For any monomial $f=x^d$, Theorem \ref{thm:(d,q-1)} precisely describes the possible intersections of its graph with horizontal lines and Theorem \ref{thm:(d-1,q-1)}  precisely describes the possible intersections of its graph with lines through the origin.  Since from the definition of the intersection distribution we do not consider intersections with the line at infinity nor vertical lines, it remains to consider only lines of the type $y=ax+b$ where $a,b$ are both non-zero.  This situation is more complicated, and a similar characterisation is not feasible.

We begin with upper bounds.
Recall that Theorem \ref{thm:ubound} gives a stronger upper bound when $a,b$ are both non-zero: the number of roots of $x^d-ax-b$ in $\F_q$ is at most $q-d$.

\emph{Lacunary} polynomials were introduced by Redei in \cite{Red}.  In \cite{SolWhiYip}, Solymosi, White and Yip describe a lacunary polynomial as a polynomial with ``a substantial gap between the degree of two consecutive terms", where the two highest terms are most commonly considered.  It is clear that our polynomial $x^d-ax-b$ can be described in this way, particularly as $d$ increases. The following lacunary polynomial result in \cite{SolWhiYip} provides us with useful insights about the intersection with lines $y=ax+b$ when $a,b \neq 0$.  In the following theorem, we present a combination of Lemma 3.1 and Lemma 3.5 of \cite{SolWhiYip}, together with ideas implicit in their proof, expressed in the language of this paper.  As in the previous cases, we are exploiting our knowledge of divisors of $q-1$.  Essentially, we show that the roots of $x^d-ax-b$ ($a,b, \neq 0$) can be partitioned into the roots of a set of distinct polynomials which depend on certain cyclotomic classes.

For a polynomial $f \in \F_q[x]$, denote by $R(f)$ the set of distinct roots of $f$ in $\F_q$.

\begin{theorem}[\cite{SolWhiYip}]\label{thm:lacunary}
Let $d\geq 1$ such that $d \mid q-1$.  Let $\alpha$ be a primitive element of $\F_q$.
\begin{itemize}
\item[(i)]  Let $h=x^{(q-1)/d - \ell}+g$ be a polynomial, where $\ell \geq 0$, $0<\mathrm{deg}(g)<(q-1)/d - \ell$ and $h(0) \neq 0$.  Then 
\begin{itemize}
    \item[(a)] For $0 \leq i \leq d-1$, $R(h) \cap C_i^d = R(x^{\ell}g+\alpha^{i(q-1)/d}) \cap C_i^d$.
\item[(b)] $|R(h)| \leq d(\ell+ \mathrm{deg}(g))$.
\end{itemize}
\item[(ii)]  Let $h=x^{m+(q-1)/d}+g$ be a polynomial, where $m \geq 0$, $0<\mathrm{deg}(g)<(q-1)/d + m$ and $h(0) \neq 0$.  Then 
\begin{itemize}
\item[(a)] For $0 \leq i \leq d-1$, $R(h) \cap C_i^d = R(\alpha^{i(q-1)/d}x^m+g) \cap C_i^d$.
\item[(b)] $|R(h)| \leq d \, \mathrm{max}(m, \mathrm{deg}(g))$.
\end{itemize}
\end{itemize}
\end{theorem}
\begin{proof}
First note that, in both cases, by assumption all roots of $h$ and $g$ are non-zero. 

For (i), since the set of non-zero roots of $h$ is the same as the set of non-zero roots of $x^{\ell}h$, we may consider $x^{\ell}h =x^{(q-1)/d} + x^{\ell}g$.  Let $r$ be a root of $x^{(q-1)/d} + x^{\ell}g$.  The $d$ cyclotomic classes of order $d$ partition $\F_q^*$, so $r \in C_i^d$ for some $0 \leq i \leq d-1$, i.e. $r=\alpha^{ad+i}$ for some $0 \leq a \leq (q-1)/d-1$.  Then $r^{(q-1)/d}+r^{\ell}g(r)=\alpha^{i(q-1)/d}+r^{\ell}g(r)$. Hence, the roots in $C_i^d$ of $x^{\ell}h$ are precisely the roots in $C_i^d$ of the polynomial $k_i=x^{\ell}g+\alpha^{i(q-1)/d}$ (where $\alpha^{i(q-1)/d} \in C_0^{(q-1)/d}$).  To obtain the upper bound, observe that since $\mathrm{deg}(g) \geq 1$, each $k_i$ ($0 \leq i \leq d-1$) has at most $\ell+\mathrm{deg}(g)$ roots.  

For (ii), let $r$ be a root of the polynomial $h=x^{m+(q-1)/d}+g$, where as above $r=\alpha^{ad+i} \in C_i^d$. Then $h(r)=\alpha^{i(q-1)/d} r^m + g(r)$. Hence, the roots in $C_i^d$ of $h$ are precisely the roots in $C_i^d$ of $k_i=\alpha^{i(q-1)/d} x^m + g$.  Since $\mathrm{deg}(g) \geq 1$ and $g(0) \neq 0$, each $k_i$ ($0 \leq i \leq d-1$) is a non-zero polynomial with at most $\mathrm{max}(m, \mathrm{deg}(g))$ roots. 
\end{proof}

When $d=1$ and $C_0^1=\F_q^*$, then Theorem \ref{thm:lacunary}(i)(a) becomes $R(x^{q-1-\ell}+g)=R(x^{\ell}g+1)$.
We now specialize to our situation.

\begin{theorem}\label{thm:lac_ax+b}
\begin{itemize}
\item[(i)] Let $f=x^{e-\ell}$ be a monomial, where $q-1=de$ ($d,e$ positive integers) and $\ell \geq 0$.  For any line $y=ax+b$ with $a,b \neq 0$, the number of points of intersection with $y=f(x)$ is at most $(\ell+1)(q-1)/e$.
\item[(ii)] Let $f=x^{e+m}$ be a monomial, where $q-1=de$ ($d,e$ positive integers) and $m \geq 1$.  For any line $y=ax+b$ with $a,b \neq 0$, the number of points of intersection with $y=f(x)$ is at most $m(q-1)/e$.
\end{itemize}
\end{theorem}
\begin{proof}
We apply Theorem \ref{thm:lacunary} with  $g=(-a)x+(-b)$ where $a,b \neq 0$. For (i), the roots in $C_i^d$ of $h=x^{e-\ell}+ax+b$ are precisely the roots in $C_i^d$ of the polynomial $ax^{\ell+1}+bx^{\ell}+\alpha^{i(q-1)/d}$. 
For (ii), the roots in $C_i^d$ of $h=x^{e+m}+ax+b$ are precisely the roots in $C_i^d$ of the polynomial $\alpha^{i(q-1)/d}x^m+ ax+b$. 
\end{proof}

Observe that, for monomials $f=x^d$ satisfying the criteria of Theorem \ref{thm:lac_ax+b}, an upper bound for $\mathrm{deg}(S_f)$ is now obtainable by considering the maximum possible number of intersections for each of the three types of lines, i.e. taking the maximum of the bounds $\mathrm{gcd}(d,q-1)$ and $ 1+\mathrm{gcd}(d-1,q-1)$ and the appropriate upper bound of Theorem \ref{thm:lac_ax+b}.

\begin{example}
Consider $\F_{25}$.  Define the primitive element $\alpha$ of $\F_{25}$ to be a root of the primitive polynomial $x^2+4x+2$. So $\alpha^2=\alpha+3$ while $\alpha^6=2, \alpha^{12}=-1$ and $\alpha^{18}=3$.  
\begin{itemize}
\item[(i)] Let $f=x^{11}$ and consider $\mathrm{deg}(S_f)$.  The upper bound from Theorem \ref{thm:ubound} is $11$; however the maximum number of intersections from horizontal lines is $\mathrm{gcd}(d,q-1)=1$ and from lines through the origin is $1+\mathrm{gcd}(d-1,q-1)=3$.  Theorem \ref{thm:lac_ax+b}(ii) with $e=12$, $d=2$ and $\ell=1$ yields an upper bound of $4$ for the number of intersections with lines $y=ax+b$ with $a,b \neq 0$. So $\mathrm{deg}(S_f) \leq 4$, and this is attainable only by lines of the last type. By Theorem \ref{thm:lacunary}(i), the (non-zero) square roots of $x^{11}+ax+b$ ($a,b \neq 0$) are the roots of $ax^2+bx+1$, while the non-square roots of $x^{11}+ax+b$ are the roots of $ax^2+bx-1$. 

Consider $x^{11}-x-\alpha^3$. Each root $r$ of this polynomial is a root of $-x^2-\alpha^3 x + r^{12}$.  Non-square roots $r$ of the quadratic $-x^2-\alpha^3 x-1$ are given by
\begin{itemize}
    \item[$\bullet$]{
        $r=\alpha^3$ since $-\alpha^6-\alpha^6-1=\alpha^{18}+\alpha^{18}-1=3+3-1=0$,
    }
    \item[$\bullet$]{
        $r=\alpha^{21}=\alpha^{-3}$ since $-\alpha^{-6}-1-1=-\alpha^{18}-2=-3-2=0$.
    }
\end{itemize}
Square roots $r$ of $-x^2-\alpha^3 x +1$ are given by
\begin{itemize}
    \item[$\bullet$]{
        $r=\alpha^{16}$ since $-\alpha^{32}-\alpha^{19}+1=\alpha^{44}+\alpha^{31}+1=3\alpha^2+2\alpha+1=3(\alpha+3)+2\alpha+1=0$,
    }
    \item[$\bullet$]{
        $r=\alpha^{20}$ since $-\alpha^{40}-\alpha^{23}+1=\alpha^{52}+ \alpha^{35}+1=\alpha^4+2\alpha^5+1=(2\alpha+2)+2(4\alpha+1)+1=0$.
    }
\end{itemize}
Hence, the line $y=x+\alpha^3$ intersects $y=x^{11}$ in four distinct points, and $\mathrm{deg}(S_f)=4$.

\item[(ii)] Let $f=x^{14}$.  Here $\mathrm{gcd}(d,q-1)=2$ and $1+\mathrm{deg}(d-1,q-1)=2$.  Using Theorem \ref{thm:lac_ax+b}(ii) with $e=12$, $d=2$ and $m=2$, there are at most four points of intersection with lines $y=ax+b$ ($a,b \neq 0$).

Consider $x^{14}-x-2$. Each root $r$ of this polynomial is a root of $r^{12}x^2-x-2$. Square roots are given by 
\begin{itemize}
    \item[$\bullet$]{
        $r=2$ since $2^2-2-2=0$,
    }
    \item[$\bullet$]{
        $r=4$ since $4^2-4-2=10=0$.
    }
\end{itemize}
Non-square roots are given by
\begin{itemize}
    \item[$\bullet$]{
        $r=\alpha^{13}=-\alpha$ since $-(-\alpha)^2-(-\alpha)-2=-\alpha^2+\alpha+3=0$,
    }
    \item[$\bullet$]{
        $r=\alpha^{17}=-\alpha^5$ since $-\alpha^{10}+\alpha^5-2=-(4 \alpha+4)+(4\alpha+1)+3=0$.
    }
\end{itemize}
Hence, $y=x+2$ intersects $y=x^{14}$ in four distinct points, and $\mathrm{deg}(S_f)=4$.
\end{itemize}
\end{example}

In a particular case, there is a source of lines that achieve the bound of Theorem \ref{thm:lac_ax+b} (and can be shown to achieve $\mathrm{deg}(S_{x^d}))$, using another concept from the theory of finite fields, namely the trace function. See \cite{LidNie} for more information about the trace.

\begin{definition}
Let $q$ be a prime power, $m \in \mathbb{N}$. Let $F=\F_{q^m}$ and let $K=\F_q$.  Let $\alpha \in K$. The trace function $\Tr_{F/K}(\alpha)$ of $\alpha$ over $K$ is defined by
$$ \Tr_{F/K}(\alpha) =\alpha + \alpha^q+ \cdots + \alpha^{q^{m-1}}.$$
$\Tr_{F/K}(\alpha)$ is always an element of $K$, and each element of $K$ has precisely $q^{m-1}$ preimages under the trace function (since the trace is a linear transformation from $F$ onto $K$).
\end{definition}

\begin{theorem}\label{thm:trace}
Let $q=p^2$ and let $f=x^{q-p}$. Then $\mathrm{deg}(S_f)=p$.
\end{theorem}
\begin{proof} We work in $\F_{p^2}$ with $f=x^{q-p}$. By Corollary \ref{cor:q-p^i}, $\mathrm{deg}(S_f) \leq p$. We show by direct construction that $\mathrm{deg}(S_f)=p$.

Consider the polynomial $x^{q-p}+x+1$. Multiplying by $x^{p-1}$, we see that this has the same (non-zero) roots as $x^p+x^{p-1}+1$. Substituting $x$ for $x^{-1}$ and rearranging transforms this equation to $x^p+x+1$.  From above, $\Tr_{\F_{p^2}/\F_p}(x)=x^p+x$, so this polynomial has precisely $p$ roots in $\F_{p^2}$, namely the set of $p$ elements whose  trace over $\F_p$ equals $-1$.  The inverses of these roots provide $p$ non-zero roots for $x^p+x^{p-1}+1$ and thus for  $x^{q-p}+x+1$. Hence, the line $y=-x-1$ intersects in $p$ distinct points with the graph of $f$.
\end{proof}
\begin{example}
Consider $\F_{25}$. Let $f=x^{q-p}=x^{20}$.  We will use the approach of Theorem \ref{thm:trace} to show that $\mathrm{deg}(S_f)=5$, so our original polynomial is $x^{20}+x+1$.  Here $\Tr_{25/5}(x)=x^5+x$.

Let $\alpha$ be a root of the primitive polynomial $x^2+4x+2$. So $\alpha^2=\alpha+3$ while $\alpha^6=2, \alpha^{12}=-1$ and $\alpha^{18}=3$. It can be calculated that the five elements with trace $-1$ are $\{\alpha^6, \alpha^8, \alpha^{13}, \alpha^{16}, \alpha^{17}\}$.  These are five distinct non-zero roots of $x^5+x+1$; their inverses $\{\alpha^7, \alpha^8, \alpha^{11}, \alpha^{16}, \alpha^{18}\}$ are five distinct non-zero roots of $x^5+x^4+1$ and hence of $x^{20}+x+1$.

To illustrate this with $\alpha^7$, note that
\[
    \alpha^{35}+\alpha^{28}+1=\alpha^{11}+\alpha^4+1=(3\alpha+2)+(2\alpha+2)+1=0,
\]
and in the original equation we have
\[
    \alpha^{140}+\alpha^7+1=\alpha^{20}+\alpha^7+1=(3\alpha+4)+2\alpha+1=0.
\]
For $\alpha^8$, we have $\alpha^{40}+\alpha^{32}+1=\alpha^{16}+\alpha^{8}+1$, which is $0$ since $\{\alpha^8,\alpha^{16}\}$ are the non-trivial cube roots of unity; in the original equation $\alpha^{160}+\alpha^8+1=\alpha^{16}+\alpha^8+1=0$.
\end{example}

Finally, we present a bound for $\mathrm{deg}(S_{x^d})$ from the literature which is independent of $d$.  The following theorem is proved by Kelley and Owen in \cite{KelOwe}.  

\begin{theorem}[\cite{KelOwe}]\label{thm:ax+b}
For a trinomial $f=x^n+ax^s+b \in \F_q[x]$, where $a,b \in \F_q$ are both non-zero, let $\delta=\mathrm{gcd}(n,s,q-1)$. Then 
$| R(f)| \leq \delta \left\lfloor \frac{1}{2}+ \sqrt{ \frac{q-1}{\delta}}  \right\rfloor$.
\end{theorem}

This leads to the following corollary.
\begin{corollary}
Let $d \geq 2$ and let $f=x^d$.  Then for any line $y=ax+b$ with $a,b \in \F_q$ both non-zero, the maximum number of points of intersection with $y=f(x)$ is $\left\lfloor \frac{1}{2}+ \sqrt{q-1} \right\rfloor$.
\end{corollary}
\begin{proof}
This is immediate upon applying Theorem \ref{thm:ax+b} to $f-ax-b$; here $\delta=1$.
\end{proof}

Combining our bounds for all three types of line, we therefore obtain the following.  Note that the final result gives similar results to Corollary \ref{cor:(d,q-1)} and Corollary \ref{cor:(d-1,q-1)}.
\begin{theorem}\label{thm:3part}
Let $d \geq 2$ and let $f=x^d$.  Then
$$ \operatorname{deg}(S_f) \leq \operatorname{max}\left(\mathrm{gcd}(d,q-1), 1+\mathrm{gcd}(d-1,q-1), \left\lfloor \frac{1}{2}+ \sqrt{q-1}  \right\rfloor\right).$$
In particular, if $\lfloor \frac{1}{2}+ \sqrt{q-1} \rfloor \leq \operatorname{max}(\mathrm{gcd}(d,q-1), 1+\mathrm{gcd}(d-1,q-1))$, then 
$$\operatorname{deg}(S_f)=\operatorname{max}(\mathrm{gcd}(d,q-1), 1+\mathrm{gcd}(d-1,q-1)).$$
\end{theorem}

We end this section by showing, in Table \ref{tab:q=25}, the values $\operatorname{deg}(S_f)$ for $f=x^d$ where $q=25$ and $1<d<q$.  The second column shows the exact value of $\operatorname{deg}(S_f)$ obtained by computation in GAP \cite{GAP4}, while the other columns indicate results from this section which may be applied to obtain the exact value or best lower/upper bounds (in several cases, multiple results may apply, but only one is given).  The final column provides an example of a line $y=ax+b$ whose intersection with $y=f(x)$ has the maximum possible $\operatorname{deg}(S_f)$ points of intersection (in most cases, these are not unique).

\begin{table}
    \begin{align*}
        \begin{array}{clllll}
        \hline
        d& \mathrm{deg}(S_{x^d})& \mbox{Exact value} & \mbox{Lower bound} & \mbox{Upper bound} & \mbox{Sample line}\\
        \hline
        2 & 2 & 2 \mbox{ (Thm \ref{thm:summarye|q-1}(ii))} &  & & y=1 \\
        3 & 3 & 3 \mbox{ (Thm \ref{thm:summarye|q-1}(ii))} &  & & y=1 \\
        4 & 4 & 4 \mbox{ (Thm\ref{thm:summarye|q-1})(ii))} &  & & y=1 \\
        5 & 5 & 5 \mbox{ (Thm \ref{thm:summarye|q-1})(iii))} &  & & y=x \\
        6 & 6 & 6 \mbox{ (Thm \ref{thm:summarye|q-1})(ii))} &  & & y=1 \\
        7 & 7 & 7 \mbox{ (Thm \ref{thm:summarye|q-1})(iii))} &  & & y=x \\
        8 & 8 & 8 \mbox{ (Thm \ref{thm:summarye|q-1})(ii))} &  & & y=1 \\
        9 & 9 & 9 \mbox{ (Thm \ref{thm:summarye|q-1}(iii))} &  & & y=x \\
        10 & 4 &   & 4 \mbox{ (Cor \ref{cor:lbound})(ii))} & 5 \mbox{ (Thm \ref{thm:3part})}  & y=x \\
        11 & 4 &   & 3 \mbox{ (Cor \ref{cor:lbound}(ii))} & 4 \mbox{ (Thm \ref{thm:lac_ax+b}(i)) } & y=x+\alpha^3 \\
        12 & 12 & 12 \mbox{ (Thm \ref{thm:summarye|q-1})(ii))} &  & & y=1 \\
        13 & 13 & 13 \mbox{ (Thm \ref{thm:summarye|q-1})(iii))} &  & & y=x \\
        14 & 4 &   & 2 \mbox{ (Cor \ref{cor:lbound})(i))} & 4 \mbox{ (Thm \ref{thm:lac_ax+b})(ii)}  & y=x+2 \\
        15 & 3 &   & 3 \mbox{ (Cor \ref{cor:lbound})} & 5 \mbox{ (Thm \ref{thm:3part})}  & y=1 \\
        16 & 8 &  8 \mbox{ (Cor \ref{cor:(d,q-1)})} &  &  & y=1 \\
        17 & 9 & 9 \mbox{ (Cor \ref{cor:(d-1,q-1)})}  &  &   & y=x \\
        18 & 6 & 6 \mbox{ (Cor \ref{cor:(d,q-1)})}  &  &  & y=1 \\
        19 & 7 & 7 \mbox{ (Cor \ref{cor:(d-1,q-1)})}  &   &  & y=x \\
        20 & 5 & 5 \mbox{ (Thm \ref{thm:trace})}   &   &  & y=-x-1 \\
        21 & 5 & 5 \mbox{ (Thm \ref{thm:summarye|q-1}(iv))} &  & & y=x \\
        22 & 4 & 4 \mbox{ (Thm \ref{thm:summarye|q-1}(iv))} &  & & y=x \\
        23 & 3 & 3 \mbox{ (Thm \ref{thm:summarye|q-1}(iv))} &  & & y=x \\
        24 & 24 & 24 \mbox{ (Thm \ref{thm:summarye|q-1}(i))} &  & & y=1 \\
        \hline
        \end{array}
    \end{align*}
    \caption{$\mathrm{deg}(S_{x^d})$ when $q=25$ (where the primitive element $\alpha$ is root of $x^2+4x+2$)}
    \label{tab:q=25}
\end{table}

\section{Algebraic and geometric characterisations of $(q+1)$-sets}\label{sec:specificvalues}

Li and Pott ask under what circumstances the non-hitting index characterises a set and, in those circumstances when it does not, what else about the intersection distribution serves to characterise it. Note that this is not a characterisation up to projective equivalence: the example of ovals shows that it is possible to have the same distribution but still be projectively inequivalent.  In this section, we look at constructions that have natural algebraic/geometric descriptions that serve to characterise them.

We first observe that knowledge of some part of the intersection distribution can be sufficient to determine all remaining values.  This can be particularly effective when the intersection distribution is ``sparse".  The following theorem shows that if $u_i(S)$ is known for $i \geq 3$, then the rest of the intersection distribution can immediately be determined.

\begin{theorem}\label{thm:u_i>2enough} 
Let $q>2$.  For a $(q+1)$-set $S\subset\pg(2,q)$, define
\begin{itemize}
\item $A_S=q^2+q+1-\sum_{i=3}^q u_i(S)$;
\item $B_S=(q+1)^2-\sum_{i=3}^q i u_i(S)$;
\item $C_S=q(q+1) -\sum_{i=3}^q i(i-1) u_i(S)$.
\end{itemize}
Then
\begin{enumerate}
\item[(i)] $u_2(S)=C_S/2$;
\item[(ii)] $u_1(S)=B_S-C_S$;
\item[(iii)] $u_0(S)=A_S-B_S+C_S/2$.
\end{enumerate}
\end{theorem}
\begin{proof}
This is immediate from Theorem \ref{thm:uisums}; note that $C_S$ is always even since $2 \mid i(i-1)$ for all $i \in \mathbb{Z}$.
\end{proof}
For sets arising from polynomials, we have an analogous result:
\begin{theorem}\label{thm:v_i>2enough} 
Let $q>2$.  For a polynomial $f\in\F_q[x]$, define
\begin{itemize}
\item $A_f=q^2-\sum_{i=3}^q v_i(f)$;
\item $B_f=q^2-\sum_{i=3}^q i v_i(f)$;
\item $C_f=q(q-1) -\sum_{i=3}^q i(i-1) v_i(f)$.
\end{itemize}
Then
\begin{enumerate}
\item[(i)] $v_2(f)=C_f/2$;
\item[(ii)] $v_1(f)=B_f-C_f$;
\item[(iii)] $v_0(f)=A_f-B_f+C_f/2$.
\end{enumerate}
\end{theorem}
\subsection{Sets contained in the union of a small number of projective lines}\label{sec:smallunion}
Li and Pott observed that the $(q+1)$-sets with the smallest non-hitting indices arise from lines of $\pg(2,q)$, or from sets that include all but one or two points on a single line of $\pg(2,q)$.  Here we summarise what is known about these cases, then extend our attention to sets contained in the union of a small number of lines.  This allows us to give an algebraic/geometric characterisation of the structure of the sets arising from various ``natural" monomials for which Li and Pott established the intersection distribution.  We interpret the sets arising from these monomials as particular examples from more general constructions of $(q+1)$-sets for which it is straightforward to determine the intersection distribution.  These examples  give rise to a range of values for $u_0$.

\subsubsection{One projective line} A $(q+1)$-set corresponding to a projective line does not arise from a polynomial as it does not have an internal nucleus; the non-zero entries of its intersection distribution are $u_{q+1}=1$ and $u_1=q^2+q$ \cite{LiPot}.
\subsubsection{One projective line and one additional point}
A $(q+1)$-set corresponding to $q$ points on a line with one further point not on that line has the second smallest value of $u_0$; the non-zero entries of their intersection distribution are $u_q=1$, $u_2=q$, $u_1=q^2-q+1$, $u_0=q-1$ \cite{LiPot}.  Such sets are projectively equivalent to those given by polynomials of degree at most one (note that any two such polynomials are projectively equivalent), and have $v_q=1$, $v_2=0$, $v_1=q^2-q$ and $v_0=q-1$.
\subsubsection{Two projective lines}\label{sec:twolines}
A $(q+1)$-set $S$ consisting of $q-1$ points on a line $\ell$ with two further points $P,Q$ on a line $m\neq \ell$ with $\ell \cap m\in  S$ has the third smallest value of $u_0$; the non-zero entries of its intersection distribution are $u_{q-1}=1$, $u_{3}=1$, $u_2=2q-4$, $u_1=q^2-3q+7$, $u_0=2q-4$ \cite{LiPot}.  This set does not have an internal nucleus, so it does not arise from a polynomial.

A $(q+1)$-set $S$ corresponding to $q-1$ points on a line $\ell$ with two further points $P,Q$ on a line $m\neq \ell$ with $\ell \cap m\notin  S$ has the fourth smallest value of $u_0$; the non-zero entries of its intersection distribution are $u_{q-1}=1$, $u_2=2q-1$, $u_1=q^2-q+4$, $u_0=2q-3$ \cite{LiPot}.  We observe that the polynomial $x^{q-1}$ gives rise to a set of this form.  It has two internal nuclei, and has $v_{q-1}=1$, $v_2=q-1$, $v_1=q^2-3q+3$, $v_0=2q-3$.  

We generalise these examples by considering all remaining $(q+1)$-sets whose points are contained in the union of two projective lines.  Note that the case described above is the only one that arises from a polynomial.  We begin with the case where the intersection of the two lines is not part of the set.

\begin{construction}\label{con:twolines1}
Let $\ell$ and $m$ be lines of $\pg(2,q)$.  For $t\in\{0,1,\dots,\lfloor(q+1)/2\rfloor-2\}$, we construct a $(q+1)$-set by taking $\lceil(q+1)/2\rceil+t$ points on $\ell\setminus\ell\cap m$ and $\lfloor(q+1)/2\rfloor-t$ points on $m\setminus \ell\cap m$.
\end{construction}
A comment on notation: in listing the intersection distribution in the statement of Theorem~\ref{thm:twolines1}, we observe that the subscripts $\lceil(q+1)/2\rceil+t$ and $\lfloor(q+1)/2\rfloor-t$ may coincide with each other, or with the subscript 2, for particular choices of $q$ and $t$. In these cases, they would not represent distinct entries in the intersection distribution. The values in question should instead be added. We have chosen to adopt this approach to avoid giving a complicated description in which all such cases are considered separately.  An analogous  approach is used in all theorems throughout this section. 
\begin{theorem}\label{thm:twolines1}
    The non-zero entries of the intersection distribution of the $(q+1)$-set $S$ arising from Construction~\ref{con:twolines1} are
    \begin{align*}
u_{\lceil(q+1)/2\rceil+t}(S)&=1, \\
u_{\lfloor(q+1)/2\rfloor-t}(S)&=1,\\
u_2(S)&=(\lceil(q+1)/2\rceil+t)(\lfloor(q+1)/2\rfloor-t),\\
u_1(S)&=q^2+q-2(\lceil(q+1)/2\rceil+t)(\lfloor(q+1)/2\rfloor-t),\\
u_0(S)&=(\lceil(q+1)/2\rceil+t)(\lfloor(q+1)/2\rfloor-t)-1.
    \end{align*}
    \end{theorem}
\begin{proof}
Let $A=\lceil(q+1)/2\rceil+t$ be the number of points on $\ell$ and $B=(\lfloor(q+1)/2\rfloor-t)$ be the number of points on $m$.  
By construction, we have $u_A=1$, $u_B=1$ and $u_i=0$ for all other $i\geq 3$.  Using the notation of Theorem~\ref{thm:u_i>2enough}, we have
\begin{align*}
C_S&=q(q-1)-A(A-1)-B(B-1),\\
&=q(q+1)-A(q-B)-B(q-A),\\
&=2AB,
\end{align*}
so $u_2=AB$.  Then
\[
u_1=B_S-C_S =(q+1)^2-A-B-2AB=q(q+1)-2AB
\]
and 
\[
u_0 =A_S-B_S+\frac{C_S}{2} =q^2+q+1-2-\big((q+1)^2-q-1\big)+AB=AB-1.\qedhere
\]
\end{proof}
We can obtain a closely related construction if we choose the points of $S$ to come from two lines, and this time include the point of intersection of the two lines.
\begin{construction}\label{con:twolines2}
Let $\ell$ and $m$ be lines of $\pg(2,q)$.  For $t\in\{0,1,\dots,\lfloor q/2\rfloor-1\}$, we construct a $(q+1)$-set by taking $\ell\cap m$, together with $\lceil q/2\rceil+t$ points on $\ell\setminus\ell\cap m$ and $\lfloor q/2\rfloor-t$ points on $m\setminus \ell\cap m$.
\end{construction}
\begin{theorem}\label{thm:twolines2}
    The non-zero entries of the intersection distribution of the $(q+1)$-set $S$ arising from 
    Construction~\ref{con:twolines2} are
    \begin{align*}
    u_{\lceil q/2\rceil+t+1}(S)&=1,\\
    u_{\lfloor q/2\rfloor-t+1}(S)&=1,\\
    u_2(S)&=(\lceil q/2\rceil+t)(\lfloor q/2\rfloor-t)\\
    u_1(S)&=q^2-2(\lceil q/2\rceil+t)(\lfloor q/2\rfloor-t) +q-1,\\
    u_0(S)&=(\lceil q/2\rceil+t)(\lfloor q/2\rfloor-t).
    \end{align*}
    \end{theorem}
\begin{proof}
By construction, we have $u_{\lceil q/2\rceil+t+1}(S)=1$,  $u_{\lfloor q/2\rfloor-t+1}(S)=1$, and $u_i=0$ for all other $i\geq 3$.  The remaining values follow immediately by application of Theorem~\ref{thm:u_i>2enough}.
\end{proof}
\subsubsection{Two lines and a point}\label{sec:twolinesonepoint}  Examining all possible $(q+1)$-sets of this form could descend into a tedious list of cases; instead, we restrict our attention to those where the additional point is an internal nucleus of the set.  These are of interest as they arise from polynomials; we first show that a monomial whose intersection distribution was found in \cite{LiPot} can thus be characterised.
\begin{proposition}
Let $q$ be an odd prime power.  For $f=x^{(q+1)/2}$, $S_f$ has the following structure: aside from the point at infinity $(0,1,0)$, the remaining points are $(0,0,1)$, plus $(q-1)/2$ additional points on each of the lines $y=x$ and $y=-x$,
\end{proposition}
\begin{proof}
Here, $f(a)=a$ if $a$ is a non-zero square in $\F_q$, $f(a)=-a$ if $a$ is a non-zero non-square in $\F_q$, and $f(0)=0$.  Hence the line $y=x$ contains $1+(q-1)/2=(q+1)/2$ points of the graph, and similarly $y=-x$ contains $(q+1)/2$ points of the graph, with the only point in common being $(0,0,1)$.
\end{proof}

\begin{construction}\label{con:twolinesonepoint}
Let $\ell$ and $m$ be distinct lines of $\pg(2,q)$ where $\ell\cap m$ does not lie on the line $z=0$.  We construct a $(q+1)$-set $S_f$ arising from some polynomial $f\in\F_q[x]$ 
by initially including the point $(0,1,0)$.  Then for each line $n$ through $(0,1,0)$, other than the line $z=0$, we include exactly one point from $\{\ell\cap n,m\cap n\}$. Without loss of generality, this results in selecting $\lceil(q-1)/2\rceil+t$ points of $\ell$ for some $t\in\{0,1,\dotsc,\lfloor(q-1)/2\rfloor -1\}$ and $\lfloor(q-1)/2\rfloor-t$ points of $m$ apart from the point $\ell\cap m$, which is also selected.    
\end{construction}
\begin{theorem}\label{thm:twolinesonepoint}
The non-zero entries of the intersection distribution of the polynomial $f$ that gives rise to the $(q+1)$-set $S_f$ described in Construction~\ref{con:twolinesonepoint} are
\begin{align*}
v_{\lceil(q-1/2)\rceil+t+1}(f)&=1,\\
v_{\lfloor(q-1/2)\rfloor-t+1}(f)&=1,\\
v_2(f)&=(\lceil(q-1)/2\rceil+t)(\lfloor(q-1)/2\rfloor-t),\\
v_1(f)&=(q-1)^2-2(\lceil(q-1)/2\rceil+t)(\lfloor(q-1)/2)\rfloor-t)+q-2,\\
v_0(f)&=q-1+(\lceil(q-1)/2\rceil+t)(\lfloor(q-1)/2)\rfloor-t).
\end{align*}
\end{theorem}
\begin{proof}
By construction, we have $v_{\lceil(q-1/2)\rceil+t+1}(f)=1$, $v_{\lfloor(q-1/2)\rfloor-t+1}(f)=1$ and $v_i=0$ for all other $i\geq 3$.  The remaining values can then be computed directly using Theorem~\ref{thm:v_i>2enough}.
\end{proof}
We note that taking $t=0$ in Theorem~\ref{thm:twolinesonepoint} recovers the intersection distribution for $x^{(q+1)/2}$ that was previously determined by Li and Pott \cite{LiPot}.  All the sets arising from this construction have an internal nucleus, and it is an open question to determine the polynomials that give rise to these sets for other values of $t$.

As in Section~\ref{sec:twolines}, we can also vary this construction by excluding the intersection of the two lines from the set.

\begin{construction}\label{con:twolinesonepoint2}
Let $\ell$ and $m$ be distinct lines of $\pg(2,q)$, other than the line $z=0$, whose intersection $\ell \cap m$ lies on $z=0$ (i.e. the affine portions of these lines are parallel).  We construct a $(q+1)$-set $S_f$ arising from some polynomial $f\in\F_q[x]$ 
by initially including the points $(0,1,0)$.  Then for each line $n$ through $(0,1,0)$, other than the line $z=0$, we include exactly one point from $\{\ell\cap n,m\cap n\}$. Without loss of generality, this results in selecting $\lceil q
/2\rceil+t$ points of $\ell$ for some $t\in\{0,1,\dotsc,\lfloor q/2\rfloor -2\}$ and $\lfloor q/2\rfloor-t$ points of $m$, all of which are distinct from $\ell \cap m$.    \end{construction}
\begin{theorem}\label{thm:twolinesonepoint2}
The non-zero entries of the intersection distribution of the polynomial $f$ that gives rise to the $(q+1)$-set $S_f$ described in Construction~\ref{con:twolinesonepoint} are
\begin{align*}
v_{\lceil q/2\rceil+t}(f)&=1,\\
v_{\lfloor q/2\rfloor-t}(f)&=1,\\
v_2(f)&=(\lceil q/2\rceil+t)(\lfloor q/2\rfloor-t),\\
v_1(f)&=q^2-q-2(\lceil q/2\rceil+t)(\lfloor q/2\rfloor-t),\\
v_0(f)&=q-2+(\lceil q/2\rceil+t)(\lfloor q/2\rfloor-t).
\end{align*}
\end{theorem}
\begin{proof}
By construction, we have $v_{\lceil q/2\rceil+t}(f)=1$, $v_{\lfloor q/2\rfloor-t}(f)=1$ and $v_i=0$ for all other $i\geq 3$.  The remaining values can then be computed directly using Theorem~\ref{thm:visums}.
\end{proof}
One way to obtain polynomials giving rise to sets from this construction is to take a subset $T\subset\F_q$ with $2\leq |T|\leq q-2$ and take its {\em indicator function} $f_T$ defined by $f_T(x)=1$ if $x\in T$ and $f_T(x)=0$ if $x\notin T$.  We observe that $f_T$ can be expressesd as $f_T(x)=\sum_{t\in T}(1-(x-t)^{q-1})$.
\subsubsection{Two lines and two points}
As in Section~\ref{sec:twolinesonepoint}, here we restrict our attention to sets arising from polynomials by requiring one of the additional points to be an internal nucleus.
The monomial $x^{(q-1)/2}$, whose intersection distribution was established in \cite{LiPot}, is an example of such a set, and the difference in this distribution according to whether $q$ is equivalent to $1$ or $3$ mod $4$ shows some of the variation in behaviour that is possible in this setting.
\begin{proposition}\label{prop:twolinesandtwopoints}  Let $q$ be an odd prime power.
For $f=x^{(q-1)/2}$, the set $S_f$ consists of the points $(0,1,0)$ and $(0,0,1)$, together with $(q-1)/2$ further points on each of the lines $y=z$ and $y=-z$.  We observe that
\begin{enumerate}
    \item{
        If $q \equiv 1 \mod 4$, then $(0,1,0)$ is an internal nucleus and $v_3(f)=0$.
    }
    \item{
        If $q \equiv 3 \mod 4$, then all lines through $(0,1,0)$ are unisecants or 3-secants and $v_3(f)=(q-1)/2$ for $q\neq 7$ or $(q-1)/2+2$ when $q=7$.
    }
\end{enumerate}
\end{proposition}
\begin{proof}
Here, $f(a)=1$ if $a$ is a non-zero square in $\F_q$, $f(a)=-1$ if $a$ is a non-zero non-square, and $f(0)=0$.  Hence the points of $S_f$ other than $(0,1,0)$ are $(0,0,1)$ plus $(q-1)/2$ further points on the line $y=z$ and $(q-1)/2$ further points on the line $y=-z$. 
We observe that when $q\neq 7$, if $S_f$ has a 3-secant, then it passes through $(0,1,0)$ and contains one further point from each of the lines $y=z$ and $y=-z$.

Consider the line $y=x$. We have the equation
\[
    x^{(q-1)/2}-x=x(x^{(q-1)/2-1}-1)=0
\]
for some $c \neq 0$.  Its non-zero solutions in $\F_q$ are the roots of $\mathrm{gcd}(x^{q-1}-1,x^{(q-1)/2-1}-1)$.  By Lemma \ref{lem:roots}, $\mathrm{gcd}(x^a-1,x^b-1)=x^{\mathrm{gcd}(a,b)}-1$.  Moreover, $\mathrm{gcd}(q-1,(q-1)/2-1)$ divides $2$; specifically, it is equal to $2$ for $q \equiv 3 \mod 4$ and equal to $1$ for $q \equiv 1 \mod 4$.  So there are three solutions ($x=0,1-1$) to $x^{(q-1)/2}-x$ precisely in the case when $q \equiv 3 \mod 4$ (and two solutions when $q \equiv 1 \mod 4$).

Moreover, for any root $r$ of $x^{(q-1)/2}-x$ in $\F_q$, $cr$ is a root of $x^{(q-1)/2}-c^{-1}x$ if and only if $c$ is a non-zero square in $\F_q$. So in the case when $q \equiv 3 \mod 4$, there are precisely $(q-1)/2$ $3$-secants given by the lines $x^{(q-1)/2}-dx$ (with $d$ a non-zero square), while when $q \equiv 1 \mod 4$ there are no $3$-secants.
\end{proof}
We can generalise the cases appearing in Proposition~\ref{prop:twolinesandtwopoints} by considering sets with varying numbers of $3$-secants.
\begin{construction}\label{con:twolinestwopoints}  
Let $q$ be an odd prime power and partition $\F_q^*$ arbitrarily into sets $Q^+$and $Q^-$, where $|Q^+|=|Q^-|=(q-1)/2$, and $-x\in Q^{-}$ if and only if $x\in Q^{+}$.  
Let $\ell$ and $m$ be the lines $y=z$ and $y=-z$ of $\pg(2,q)$ respectively.  Let $0\leq t\leq \lfloor (q-1)/2\rfloor-2$, and let $c$ be an integer between $0$ and $\lfloor (q-1)/2\rfloor-t$ for which $\lceil(q-1)/2\rceil+t-c$ is even.  We construct a $(q+1)$-set $S_f$ arising from some polynomial $f\in\F_q[x]$ that has $\lceil(q-1)/2\rceil+t$ points on $\ell$ and $\lfloor (q-1)/2\rfloor-t$ points on $m$, and which has $c$ 3-secants that are distinct from $\ell$ and $m$. 
\begin{itemize}
\item We initially include the points $(0,0,1)$ and $(0,1,0)$ in $S_f$.
\item We pick $c$ elements of $Q^+$ and for each such element $\theta$ we add the points $(\theta,1,1)\in \ell$ and $(-\theta,-1,1)\in m$ to $S_f$.
\item We pick $(\lceil(q-1)/2\rceil+t-c)/2$ further elements of $Q^+$ and for each such element $\theta$ we add  the points $(\theta,1,1)\in \ell$ and $(-\theta,1,1)\in \ell$ to $S_f$. 
\item  The number of elements remaining in $Q^+$ is
\[
    \frac{q-1}{2}-c-\frac{\lceil(q-1)/2\rceil+t-c}{2}=\frac{\lfloor(q-1)/2\rfloor-t-c}{2}.
\]
For each such element $\theta$, we add the points $(\theta,-1,1)\in m$ and $(-\theta,-1,1)\in m$ to $S_f$. 
\end{itemize}
\end{construction}
\begin{theorem}\label{thm:twolinestwopoints}
The non-zero entries of the intersection distribution of the polynomial $f$ that gives rise to the $(q+1)$-set $S_f$ described in Construction~\ref{con:twolinesonepoint} are
\begin{align*}
v_{\lceil(q-1)/2)\rceil+t}(f)&=1,\\
v_{\lfloor(q-1)/2)\rfloor-t}(f)&=1,\\
v_3(f)&=c,\\
v_2(f)&=(\lceil(q-1)/2\rceil+t)(\lfloor(q-1)/2\rfloor-t)-3c+q-1,\\
v_1(f)&=1+3c+(q-2)(q-1)-2(\lceil(q-1)/2\rceil+t)(\lfloor(q-1)/2\rfloor-t),\\
v_0(f)&=2q-4-c+(\lceil(q-1)/2\rceil+t)(\lfloor(q-1)/2\rfloor-t).
\end{align*}
\end{theorem}
\begin{proof}
By construction, we have $v_{\lceil(q-1)/2\rceil+t}(f)=1$, $v_{\lfloor(q-1)/2\rfloor-t}(f)=1$, $v_3=c$, and $v_i\geq 0$ for all other $i\geq 3$. The result follows from Theorem~\ref{thm:v_i>2enough}.
\end{proof}

\subsection{The non-hitting spectrum}\label{subsec:spec}
Given a prime power $q$, Li and Pott define its {\em non-hitting spectrum} $\spec(q)$ to be the set of values of the non-hitting index $u_0(S)$ attained by $(q+1)$-sets $S$ in $\pg(2,q)$.  They give the spectra for $q=2,3,5$, and determine the five smallest and three largest values for $\spec(7)$.  Furthermore, they determine that the smallest entries of $\spec(q)$ are $0,q-1,2q-4,2q-3$, and show that sets with those non-hitting indices are uniquely characterised by their non-hitting index. In this section, we find the next 5 values of the spectrum, and show that they are no longer uniquely characterised by their non-hitting index.  We then completely determine $\spec(q)$ for $q=7,8,9$.

Li and Pott's Theorem 2.12 in \cite{LiPot} gives a lower bound on the non-hitting index of a $(q+1)$-set of degree $n$, which can be written as
\begin{align}
    u_0\geq(n-1)(q+1-n).\label{eq:lipottbound}
\end{align}
The non-hitting indices of the sets obtained by our Construction~\ref{con:twolines2} 
are tight with respect to this bound, and occur for any degree that is greater than $\lceil q/2\rceil$.  This shows that the bound \eqref{eq:lipottbound} is tight for degrees $\lceil q/2\rceil+1\leq n\leq q+1$.  However, for small degrees this is not the case in general, and it is possible to obtain stronger lower bounds. We first prove a lemma concerning sets of small degree.

\begin{lemma}\label{lem:u_0expressions}
Let $S$ be a $(q+1)$-set of $\pg(2,q)$. 
\begin{itemize}
\item[(a)] If $S$ has degree at most $3$, then
\begin{itemize}
\item[(i)] $u_1 = (3q^2-q+2)/2-3u_0$;
\item[(ii)] $ u_2 = 3u_0-q(q-2)$;
\item[(iii)] $u_3 = q(q-1)/2-u_0$.
\end{itemize}
\item[(b)] If $S$ has degree at most $4$ and $u_3(S)=c$, then
\begin{itemize}
\item[(i)] $u_1 = \frac{1}{3}( 4q^2-q+3+c-8u_0)$;
\item[(ii)] $u_2=2u_0-c-q(q-3)/2$;
\item[(iii)] $u_4=\frac{1}{3}(q(q-1)/2-c-u_0)$.
\end{itemize}
\end{itemize}
\end{lemma}
\begin{proof}
This follows from Theorem~\ref{thm:u_i>2enough}.
\end{proof}

\begin{theorem}\label{thm:betterbound}
Let $S$ be a $(q+1)$-set of $\pg(2,q)$. If $S$ has degree at most $3$, then
\begin{align*}
u_0(S)\geq \frac{q(q-2)}{3}.
\end{align*}
If $S$ has degree at most $4$ and $u_3(S)=c$, then
\begin{align*}
u_0(S)\geq \frac{q^2-3q+2c}{4}.   
\end{align*}
\end{theorem}
\begin{proof}
By definition, each $u_i$ is a non-negative integer.   The result follows by combining the inequality $u_2 \geq 0$ with the expressions for $u_2$ from Lemma  \ref{lem:u_0expressions}(a) and (b).  We observe that in the degree 3 case the bound is tight if and only if $u_2=0$, in which case the points of $S$ together with the set of triples of points occurring on 3-secants forms a Steiner triple system.  In particular, this would require $q$ to be even.
\end{proof}

Li and Pott observe that a $(q+1)$-set attains $u_0=2q-4$ if and only if $n=q-1$ and it consists of $q-1$ points on a line $\ell$, together with two further points on a line $m$ that intersects $\ell$ in a point of the set.  It attains $u_0=2q-3$ if and only if $n=q-1$ and it consists of $q-1$ points on a line $\ell$, together with two further points on a line $m$ that intersects $\ell$ in a point that is not part of the set.  This is the largest value of $u_0$ attainable by a $(q+1)$-set of degree at most $q-1$.  
\begin{lemma}\label{lem:degq-2}
For $q>7$, there is no $(q+1)$-set $S$ with $2q-2\leq u_0(S)\leq 3q-10$.
\end{lemma}
\begin{proof}
If a $(q+1)$-set has degree $q-1$ or greater, then it satisfies $u_0\leq 2q-3$.  If it has degree $n$ with $q-2\geq n>3$, then by \eqref{eq:lipottbound} it satisfies $u_0\geq (q-3)\cdot 3=3q-9$.

By Theorem~\ref{thm:betterbound}, we have $u_0\geq q(q-2)/3$ for sets of degree at most $3$.  We observe that
\[
    \frac{q(q-2)}{3}-(3q-9)=\frac{q^2-11q+27}{3}.
\]
Now for $q>(11+\sqrt{13})/2\approx 7.3$, this expression is positive, i.e. no set of degree at most 3 has a non-hitting index less than $3q-9.$
\end{proof}
\begin{lemma}
A $(q+1)$-set $S$ in $\pg(2,q)$ of degree $q-2$ has $u_0(S)\in\{3q-9,3q-8,3q-7,3q-6\}$.
\end{lemma}
\begin{proof}
The possible structures for a $(q+1)$-set in $\pg(2,q)$ of degree $q-2$ are:
\begin{itemize}
\item $q-2$ points on a line $\ell$ and 3 further points on a line $m$ that meets $\ell$ in a point of the set ($u_0=3q-7$ by Theorem~\ref{thm:twolines1});
 \item $q-2$ points on a line $\ell$ and 3 further points on a line $m$ that meets $\ell$ in a point not in the set ($u_0=3q-9$ by Theorem~\ref{thm:twolines2});
\item $q-2$ points on a line $\ell$ plus a  triangle whose points do not lie on $\ell$, where $0,1,2$ or $3$ of the sides of the triangle intersect $\ell$ in a point of the set.  ($u_0=3q-6,3q-7,3q-8,3q-9$ by Theorem~\ref{thm:u_i>2enough}).\qedhere
\end{itemize}
\end{proof}
We observe that $u_0=3q-7$ and $u_0=3q-9$ can each be obtained by sets where the three points not on the line $\ell$ are either collinear or not, showing that these non-hitting indices do not uniquely characterise the corresponding set.  
\begin{lemma}
For $q\geq 16$, there is no $(q+1)$-set $S$ in $\pg(2,q)$ with $3q-5\leq u_0(S)\leq 4q-17$.
\end{lemma}
\begin{proof}
By \eqref{eq:lipottbound}, $4q-16$ is the smallest possible non-hitting index for a set of degree at least $q-3$, and sets of degree greater than $4$ do not attain smaller values of $u_0$.  By Theorem~\ref{thm:betterbound}, a set of degree at most $4$ has $u_0\geq q(q-3)/4$, which is greater than $4q-16$ when $q \geq 16$. 
\end{proof}
Thus, we see that when $q\geq 7$, the first five entries of $\spec(q)$ are
\[
    0,q-1,2q-4,2q-3,3q-9,
\]
and when $q\geq 16$, the first nine entries are
\[
    0,q-1,2q-4,2q-3,3q-9,3q-8,3q-7,3q-6,4q-16.
\]

We now turn our attention to $q=7,8,9$.  The constructions presented in Section~\ref{sec:smallunion} provide a range of values for $\spec(q)$.  We summarise this contribution in Table~\ref{tab:specoddq} for odd $q$ and Table~\ref{tab:specevenq} for even $q$.

\begin{table}[h]
\begin{align*}
\begin{array}{ccc}
\hline
\text{construction}&v_0&t\\
\hline
\ref{con:twolines1},\ref{con:twolinesonepoint}&\left(\frac{q+1}{2}\right)^2-1-t^2&\{0,\dotsc, \frac{q-3}{2}\}\\
\ref{con:twolines2}&\left(\frac{q+1}{2}\right)^2-\frac{q+1}{2}-t-t^2&\{0,\dotsc,\frac{q-3}{2}\}\\
\ref{con:twolinesonepoint2}&\left(\frac{q+1}{2}\right)^2-1-t-t^2&\{0,\dotsc, \frac{q-5}{2}\}\\
\ref{con:twolinestwopoints}&\left(\frac{q+1}{2}\right)^2+q-4-a-t^2&\{0,\dotsc,\frac{q-5}{2}\}\\
&&a\leq \frac{q-1}{2}-t,\ \frac{q-1}{2}+t-a\text{ is even}\\
\hline
\end{array}
\end{align*}
\caption{Contributions to $\spec(q)$ from constructions in Section~\ref{sec:smallunion}
when $q$ is odd.}\label{tab:specoddq}
\end{table}
\begin{table}[h]
\begin{align*}
\begin{array}{ccc}
\hline
\text{construction}&v_0&t\\
\hline
\ref{con:twolines1},\ref{con:twolinesonepoint}&\left(\frac{q}{2}\right)^2+\frac{q}{2}-t-t^2-1&\{0,\dotsc,\frac{q}{2}-2\}\\
\ref{con:twolines2}&\left(\frac{q}{2}\right)^2-t^2&\{0,\dotsc,\frac{q}{2}-1\}\\
\ref{con:twolinesonepoint2}&q-2+\left(\frac{q}{2}\right)^2-t^2&\{0,\dotsc,\frac{q}{2}-2\}\\
\hline
\end{array}
\end{align*}
\caption{Contributions to $\spec(q)$ from constructions in Section~\ref{sec:smallunion}
when $q$ is even.}\label{tab:specevenq}
\end{table}  

By using the FinInG package in GAP \cite{fining} to generate random $(q+1)$-sets that contain the fundamental quadrangle and compute their intersection distribution, we were able to completely determine $\spec(7)$, and find additional entries of $\spec(8)$ and $\spec(9)$ (see Table~\ref{tab:spectra}).  The bold entries in the spectra for $q=7,8,9$ presented in Table~\ref{tab:spectra} correspond to values obtained by sets described in Section~\ref{sec:smallunion}.

\begin{table}
    \begin{align*}
        \begin{array}{cll}
        \hline
        q&\spec(q)&\\
        \hline
        2&0,1&\cite{LiPot}\\
        3&0,2,3&\cite{LiPot}\\
        4&0,3,4,5,6&\cite{LiPot}\\
    5&0,4,6,7,8,9,10&\cite{LiPot}\\
    7&0,\mathbf{6},\mathbf{10},\mathbf{11},\mathbf{12},\mathbf{13},\mathbf{14},\mathbf{15},\mathbf{16},17,\mathbf{18},19,21&\cite{LiPot}, \text{current work}\\
    8&\mathbf{0},\mathbf{7},\mathbf{12},\mathbf{13},\mathbf{15},\mathbf{16},\mathbf{17},\mathbf{18},\mathbf{19},20,\mathbf{21},\mathbf{22},23,24,25,28&\text{current work}\\
    9 & \mathbf{0},\mathbf{8},\mathbf{14},\mathbf{15},\mathbf{18},19,\mathbf{20},21,\mathbf{22},\mathbf{23},\mathbf{24},25,\mathbf{26},27,\mathbf{28},29,\mathbf{30},31,32,33,36&\text{current work}\\
        \hline
        \end{array}
    \end{align*}
    \caption{The non-hitting spectra of small prime powers.  Numbers in bold arise from sets described in Section~\ref{sec:smallunion}, other values arise from \cite{LiPot} or were found through computation.}
    \label{tab:spectra}
\end{table}
In the case of $q=9$, Theorem 2.9 of Li and Pott implies that a $10$-set that is not a $10$-arc has size at most $34$, and only attains that size if it contains a maximal $8$-arc.  Up to projective equivalence, there are two distinct $8$-arcs in $\pg(2,9)$ \cite{Hirschfeld}.  By starting with each of these arcs in turn and checking all ways of adding two additional points, we were able to rule out the existence of $10$-sets with non-hitting index $34$ in $\pg(2,9)$.
We have therefore accounted for every attainable value of non-hitting index with $q=7$, $q=8$ and $q=9$.

Our experiments showed that values in the middle of the spectrum were comparatively easy to generate using random examples, which is not surprising.

\subsection{Understanding cubics via geometric and algebraic viewpoints}

We end by showing how the geometric and algebraic viewpoints outlined in this paper provide elegant ways to understand the situation for cubic polynomials, providing alternative approaches to that of \cite{KyuLiPot}.  As a result of Theorem \ref{thm:v_i>2enough}, for a cubic $f$ it is sufficient to evaluate just one of $\{v_0,v_1,v_2,v_3\}$; the rest of the intersection distribution follows.

\subsubsection{An approach using known results for plane cubic curves}
The techniques described in this section involve applying standard results relating to algebraic curves.  We introduce just enough terminology to describe the approach, but refer the interested reader to \cite{Hirschfeld} for a fuller explanation of the relevant concepts.

The set $S_f$, where $f=a_3x^3+a_2x^2+a_1x+a_0\in\F_q[x]$ with $a_3\neq 0$, can be viewed as the set of rational points of the cubic curve described by the homogeneous equation
\begin{align*}
-yz^2+a_3x^3+a_2x^2z+a_1xz^2+a_0z^3=0.    
\end{align*}
Let $F\in \F_q[x,y,z]$ denote the homogeneous cubic $-yz^2+a_3x^3+a_2x^2z+a_1xz^2+a_0z^3$, and let $\mathcal{C}_F$ denote the corresponding curve.  The polynomial $F$ is irreducible over the algebraic closure of $\F_q$, hence $\mathcal{C}_F$ can be said to be {\em absolutely irreducible}.

We observe that \begin{align*}
\frac{\partial F}{ \partial x}&=3a_3x^2+2a_2xz+a_1z^2,\\
\frac{\partial F}{\partial y}&=-z^2,\\
\frac{\partial F}{\partial z}&=-2yz+a_2x^2+2a_1xz+3a_0z^3.
\end{align*}
These (formal) partial derivatives all vanish at the point $(0,1,0)\in \mathcal{C}_F$, hence this is a {\em singular point} of $\mathcal{C}_F$; the other points of $\mathcal{C}_F$ are non-singular.  In particular, $(0,1,0)$ is a {\em cusp} of $\mathcal{C}_F$, as $F$ has no terms involving $y^3$ or $y^2$,  and the only term involving $y$ is a scalar multiple of $y$ times the square $z^2$ \cite{Hirschfeld}.
\begin{lemma}
\begin{itemize}
\item When $p\neq 3$, for every cubic polynomial $f\in\F_{p^m}[x]$ the set $S_f$ is projectively equivalent to the set of rational points of the cubic curve with equation 
\begin{align*}
-yz^2+x^3=0.
\end{align*}
\item For a cubic polynomial $f\in \F_{3^m}[x]$, there are two projectively inequivalent possibilities for the set $S_f$: it is either projectively equivalent to the set of rational points of the cubic with equation
\begin{align*}
-yz^2+x^3=0,    
\end{align*}
or it is projectively equivalent to the set of rational points of the curve with equation
\begin{align*}
    -yz^2+x^3+x^2z=0.
\end{align*}
\end{itemize}
\end{lemma}
\begin{proof}
This follows immediately from the classification of plane cubics with a cusp given in Lemma 11.24 of \cite{Hirschfeld} once one applies the appropriate choice of variables.
\end{proof}
A line of $\pg(2,q)$ intersects an absolutely irreducible cubic in at most three points.  It follows that in order to determine the intersection distribution for a cubic polynomial $f$, it suffices to determine $v_2(f)$: we have $v_i(f)=0$ for $i\geq 4$, and the remaining $v_i$ can be calculated from $v_2(f)$ using Theorem \ref{thm:visums}.

If a univariate cubic polynomial over $\F_q[x]$ has two roots in $\F_q$, it follows that the third root also lies in $\F_q$.  Thus, if two of the points of intersection between a line and the curve $\mathcal{C}_F$ are rational, then the third point of intersection is also rational. Hence, lines that meet $S_f$ in two points have a double intersection with $\mathcal{C}_F$ at some point $P$, and a single intersection with $\mathcal{C}_F$ at another point $Q$.  Such a line is known as a {\em tangent} to $\mathcal{C}_F$ at $P$, and every non-singular point $P\in \mathcal{C}_F$ has a unique tangent at $P$.  Note that in the case where $Q=P$, the tangent at $P$ has a triple point of intersection with $\mathcal{C}_F$ at $P$, in which case $P$ is referred to as a {\em point of inflection.}  As $\mathcal{C}_F$ has $q$ non-singular points, we deduce that the number of lines intersecting $S_f$ in two points (ignoring the lines through $(0,1,0)$, which are all tangents at $(0,1,0)$) is $q$ minus the number of (non-singular) points of inflection of $\mathcal{C}_F$.

Table 11.4 of \cite{Hirschfeld}\footnote{These values can be readily derived by computing the {\em Hessian} of $\mathcal{C}_F$ when working in odd characteristic, or in even characteristic by observing that in $\pg(2,2^m)$, every tangent to the curve $-yz^2+x^3$ at a non-singular point also passes through the point $(0,0,1)$, hence this is the only point of inflection.} tells us that the cubic of equation $-yz^2+x^3$ in $\pg(2,p^m)$ has one rational point of inflection when $p\neq 3$ and $q$ rational points of inflection when $p=3$.  The cubic $-yz+x^3+x^2z$ in $\pg(2,3^m)$ has no points of inflection.
\begin{align*}
\begin{array}{cccc}
\hline
\text{characteristic}&f&\text{number of inflections}& v_2(f)\\
\hline
p\neq 3 & x^3&1&q-1\\
3&x^3&q&0\\
3&x^3+x^2&0&q\\
\hline
\end{array}
\end{align*}
These agree with the values for $v_2(f)$  obtained in Theorem 1.6 of \cite{KyuLiPot}; the remaining values of the intersection distribution then follow from Theorem \ref{thm:v_i>2enough}.

\subsubsection{Using the theory of irreducible polynomials to explain intersection distribution of cubics}

In this section, we show how the intersection distribution for cubic polynomials can be directly obtained from known results for irreducible polynomials of fixed trace.  

We begin with a result from the finite field literature.  The formula $\frac{1}{n} \sum_{d \mid n} \mu(d) q^{n/d}$ for the number of monic irreducible polynomials of degree $n$ over $\F_q$ is well-known (see \cite{LidNie}). The formula for monic irreducibles with first coefficient fixed (i.e. fixed trace) is originally due to Carlitz (for $\gamma \neq 0$), though here we use the version by Yucas \cite{Yuc}.
\begin{theorem}\label{thm:irreds}
Let $q=p^s$.  Let $\mu$ denote the M\"{o}bius function, i.e. $\mu(n)$ equals $1$ if $n=1$, $(-1)^r$ if $n=p_1 \ldots p_r$ for distinct primes $p_i$, and $0$ if $p_i^2 \mid n$ for some prime $p_i$.  Then
\begin{itemize}
\item[(i)] Let $\gamma \in \F_q^*$. Let $n=p^k m$ where $p \nmid m$.  The number of monic irreducible monic polynomials $x^n+a_{n-1}x^{n-1}+ \cdots a_1x+a_0$ of degree $n$ over $\F_q$ with $a_{n-1}=\gamma$ is given by
$$N_{\gamma}(q,n)= \frac{1}{qn} \sum_{d \mid m} \mu(d) q^{n/d}.$$
\item[(ii)] Let $\gamma=0$. Let $n=p^k m$ where $p \nmid m$.  The number of monic irreducible monic polynomials $x^n+a_{n-1}x^{n-1}+ \cdots a_1x+a_0$ of degree $n$ over $\F_q$ with $a_{n-1}=0$ is given by
$$N_0(q,n)=\frac{1}{qn} \sum_{d \mid m} \mu(d) q^{n/d}- \frac{\epsilon}{n} \sum_{d \mid m} \mu(d) q^{n/{pd}},$$
where $\epsilon$ equals $1$ if $k>0$ and $0$ if $k=0$.
\end{itemize}
\end{theorem}

Theorem \ref{thm:irreds} allows us to immediately establish the non-hitting index (and hence the complete intersection distribution) for cubic polynomials. Recall that, from our prior simplifications, it suffices to consider a cubic of the form $x^3+ax^2$.

\begin{theorem}
Let $q=p^s$. Let $f=x^3+ax^2$ be a cubic polynomial over $\F_q$.
\begin{itemize}
\item[(i)] If $p \neq 3$, then $v_0(f)=\frac{q^2-1}{3}$.
\item[(ii)] If $p=3$ and $a \neq 0$, then $v_0(f)=\frac{q^2}{3}$.
\item[(iii)] If $p=3$ and $a=0$, then $v_0(f)=\frac{q(q-1)}{3}$.
\end{itemize}
\end{theorem}
\begin{proof}
By definition, $v_0(f)$ is the number of pairs $(c,d) \in \F_q^2$ such that $x^3+ax^2-cx-d$ has no roots, i.e. the number of monic cubic polynomials over $\F_q$ with $x^2$-coefficient equal to $a$ with no roots in $\F_q$.  Since any factorization of a cubic must have at least one linear factor, this is equal to the number of monic irreducible cubic polynomials with $x^2$-coefficient equal to $a$, i.e. equal to $N_a(q,3)$.

Applying parts (i) and (ii) of  Theorem \ref{thm:irreds} with $n=3$ and $p \neq 3$, we see that for any $a \in \F_q$, we have
\[
    N_{a}(q,3) = \frac{1}{3q}(q^3-q)=\frac{q^2-1}{3},
\]
i.e. part (i) follows immediately. 
For part (ii), we apply part (i) of  Theorem \ref{thm:irreds} to see that for $n=p=3$, when $a \neq 0$, we have
\[
    N_{a}(q,3) = \frac{1}{3q}q^3=\frac{q^2}{3}.
\]
For part (iii), we apply part (ii) of  Theorem \ref{thm:irreds} to see that in this case
\[
    N_{0}(q,3) = \frac{1}{3q}q^3-\frac{1}{3q}q^2=\frac{q(q-1)}{3}.\qedhere
\]
\end{proof}

These agree with the values for $v_0(f)$  obtained in Theorem 1.6 of \cite{KyuLiPot}; the remaining values of the intersection distribution  then follow from Theorem \ref{thm:v_i>2enough}.

\section{Conclusion and open problems}
We have obtained a clearer picture of the behaviour of projective equivalence of polynomials and developed understanding of the degree of $(q+1)$-sets arising from polynomials.  We have characterised the $(q+1)$-sets of degree at least $q-2$ and determined the entries of the non-hitting spectrum up to $4q-16$.  However, there are still many interesting questions that remain regarding the intersection distribution.  
\begin{description}
\item[Open problem 1] From the polynomial viewpoint, it would be of interest to describe more precisely the intersection between monomials and $y=ax+b$ lines with $a,b \neq 0$ (particularly in relation to degree), and to move beyond the case of monomials. 

\item[Open problem 2] The non-hitting spectrum can have an arbitrarily large number of gaps on the lower end but for small values of $q$, it only has one at the top end. It looks like an interesting problem to determine whether or not there is more than one gap the top end of the spectrum.
\item[Open problem 3] From experiments with random sets of size $(q+1)$ using the FinInG package in GAP, the average value of the non-hitting index seems to lie around $0.735\, q(q-1)/2$, i.e. roughly three quarters of the maximal possible value. This also suggests that there should be less gaps in the non-hitting spectrum at the top end compared to the lower end. We ask whether there is a constant $C$ such that the expected value of the non-hitting index is $C\, q(q-1)/2$ as $q \to \infty$ and what its value is. This question can also be asked for the sets $S_f$, i.e. for functions from $\mathbb{F}_q$ to $\mathbb{F}_q$, and it looks like the constant is the same. Investigations like this have applications to random functions.
\item[Open problem 4] Resolve whether there exist projectively equivalent polynomials that are not CCZ-equivalent.
\item[Open problem 5] Determine all polynomials giving rise to the sets of Construction~\ref{con:twolinesonepoint}.
\item[Open problem 6] Establish tight lower bounds for the non-hitting index of $(q+1)$-sets of degree $n<\lceil q/2\rceil+1$.
\end{description}

\subsection*{Acknowledgements}

The second author was partially funded by the Deutsche Forschungsgemeinschaft (DFG, German Research Foundation) – Project-ID 491392403 – TRR 358. He thanks the DAAD (German Academic Exchange Service) for funding a research visit to Birkbeck, University of London, and the University of St Andrews, during which much of this project occurred.

\bibliography{intersectiondist}

\section*{Appendix: GAP code}

Here, we give GAP code to compute the intersection distribution using the FinInG package \cite{fining}.

\begin{lstlisting}
#Construct the finite projective plane PG(2,q) using FinInG for the sample value q=9
q:=9;
plane:=PG(2,q);
points:=Elements(Points(plane));
lines:=Elements(Lines(plane));

#Function for computing the intersection distribution of a set of projective points with respect to a set of projective lines
IntersectionDistribution:=function(set,lines)
	local q,intersections;
	q:=Size(BaseField(lines[1]));
    #Compute the intersection size for every line
	intersections:=List(lines,l->Number(set,x->x*l));
    #Count how often each intersection size occurs
	return List([0..q+1],i->Number(intersections,x->x=i));
end;

#Intersection distribution of a polynomial f over GF(q) using the affine definition
#Sample usage: ComputePoly(x->x^3,11)
ComputePoly:=function(f,q)
    local elements,lines,intersections,distribution;
    elements:=Elements(GF(q));
    lines:=Tuples(elements,2);
    #Compute the intersections with the affine lines
    intersections:=List(lines,l->Number(elements,x->f(x)=l[1]*x+l[2]));
    #Count how often each intersection size occurs
    distribution:=List([0..q],i->Number(intersections,x->x=i));
    Print("Degree: ",Maximum(PositionsProperty(distribution,x->x>0))-1,"\n");
    return distribution;
end;

#Create a random set of k distinct elements from the set points
RandomSet:=function(points,k)
	local S,p;
	S:=[];
	while Length(S) < k do
    #Draw an element p from the set points uniformly at random
		p:=Random(points);
    #AddSet ensures that p is only added if it is not a duplicate
		AddSet(S,p);
	od;
	return S;
end;

#Intersection distribution of a random set of q+1 projective points
IntersectionDistribution(RandomSet(points,q+1),lines);
\end{lstlisting}

\end{document}